\documentclass[11pt]{amsart}
\usepackage{latexsym,amsfonts,amssymb,graphicx,amsmath,color}
\usepackage{epstopdf}
\usepackage{amssymb}
\usepackage{enumerate}
\usepackage{enumitem}
\usepackage{float}
\numberwithin{equation}{section}
\newcommand{\n}[1]{\boldsymbol{#1}}

\newtheorem{theorem}{\bf Theorem}
\newtheorem{lemma}{\sc Lemma}
\newtheorem{proposition}{\sc Proposition}

\newtheorem{Def}{Definition}[section]
\newtheorem{hyp}{\sc Hypothesis}
\newtheorem{rem}{\sc Remark}
\renewcommand{\theenumi}{\thesubsection.\arabic{enumi}}

\title[Higher-order fem
for elliptic problems with interfaces]{Higher-order finite element methods
for elliptic problems with interfaces}

\author{Johnny Guzm\'an\textsuperscript{1}, \, Manuel S\'anchez-Uribe\textsuperscript{2}\, and \, Marcus Sarkis \textsuperscript{3} }
\address{\textsuperscript{1} Division of Applied Mathematics, Brown University, Providence, RI 02912, USA}
\email{\textsuperscript{1}johnny\_guzman@brown.edu}
\address{\textsuperscript{2} Division of Applied Mathematics, Brown University, Providence, RI 02912, USA}
\email{\textsuperscript{2}manuel\_sanchez\_uribe@brown.edu}
\address{\textsuperscript{3} Department of Mathematical Sciences at Worcester Polytechnic Institute, 100 Institute Road, Worcester, MA 01609, USA.  }
\email{\textsuperscript{2}msarkis@wpi.edu}

\thanks{Mathematics Subject Classification:  65N30, 65N15.}
\begin{document}

\date{}
\begin{center}
\begin{minipage}[b]{\textwidth}
\begin{center}
\scriptsize{ An earlier version of this paper appeared  on November 13, 2014 in http://www.brown.edu/research/projects/scientific-computing/reports/2014 }
\end{center}
\end{minipage} 
\end{center}

\begin{abstract}
\noindent
We present higher-order piecewise continuous finite element methods for solving a class of interface problems in two dimensions. The method is based on correction terms added to the right-hand side in the standard variational formulation of the problem. We prove optimal error estimates of the methods on general quasi-uniform and shape regular meshes in maximum norms. In addition, we apply the method to a Stokes interface problem, adding correction terms for the velocity and the pressure, obtaining optimal convergence results.

\vskip10truept
\noindent
{\it Keywords}: Interface problems, finite elements, pointwise estimates.
\vspace{3pt}
\end{abstract}

\maketitle

\section{Introduction}\label{section1}

In this paper we continue the work started in \cite{GSS2014} and consider higher-order piecewise continuous finite element approximations to the following interface problem: Let $\Omega\subset \mathbb{R}^2$ be a polygonal domain with an  immersed smooth, closed interface $\Gamma$ such that $\overline{\Omega}=\overline{\Omega}^{-} \cup \,\overline{\Omega}^{+}$ and $\Gamma$ encloses $\Omega^{-}$. Consider the problem
\begin{subequations}\label{Problem}
\begin{align}
 -\Delta u &= f \qquad\mbox{in }\Omega, \label{Problem:a} \\
  u &= 0 \qquad \mbox{on } \partial \Omega, \label{Problem:b}\\
  \left[u\right]&=0 \qquad \mbox{on }\Gamma, \label{Problem:c} \\
  \left[D_{\n{n}} u  \right]&= \beta \qquad \mbox{on }\Gamma, \label{Problem:d}
\end{align}
\end{subequations}
where the jumps across the interface $\Gamma$ are defined as
\begin{equation*}
\left[u\right] = u^+-u^-,\quad\left[D_{\n{n}} u  \right]=D_{\n{n}^{-}} u^{-} +D_{\n{n}^{+}} u^{+}  =\nabla u^{-} \cdot \n{n}^{-}+ \nabla u^{+} \cdot \n{n}^+.
\end{equation*}
Here we denote by $u^{\pm}=u|_{\Omega^{\pm}}$ and $\n{n}^{\pm}$ is the unit outward pointing normal to $\Omega^{\pm}$.

Numerically, the problem is to find an approximate solution on meshes not aligned with the interface, that is, we allow the interface to cut elements. In this context, the finite difference methods by Peskin \cite{MR0490027,MR2009378} (i.e. \textit{immersed boundary method}) and by LeVeque and Li \cite{MR1443639} (i.e. \textit{immersed interface method}) are the most renowned. Both methods were developed for more involved
problems and for lower-order finite differences techniques. The aim of this paper
is to develop higher-order methods
based on finite element methods  and to establish a priori pointwise
error estimates. We consider the Poisson interface problem  \eqref{Problem} and the
Stokes interface problem \eqref{StokesProblem} which will be fundamental
towards developing very accurate methods with a
rigorous finite element analysis for more involved problems.

Naturally, finite element versions of the methods above have appeared; see for example \cite{MR2001876, MR2740492, MR2377272, MR2660312, MR3218337, MR2917495, MR2728820, MR2738930, MR2899249, MR3051411, MR3163976, MR2684351, MR2145387, MR2677772, MR3072968,GSS2014}. In our recent work \cite{GSS2014} we derived a piecewise linear finite element method for the above problem and proved it is second-order accurate. The attractive feature of the method in   \cite{GSS2014} is that only the right-hand side needs to be modified, which is one of the advantages also of the immersed boundary method (see \cite{MR0490027}, \cite{MR2001876}) and  immersed interface method (see \cite{MR1286215}, \cite{MR2244270}, \cite{MR2740492}) for the above problem. Moreover, the correction term added in \cite{GSS2014} is only based on the edges intersecting the interface, this is due to the fact that test functions, linear polynomials, are harmonic.  However, it seems that the approach used in \cite{GSS2014} cannot be generalized to higher order approximations. Therefore, in this paper the approach we take is based on correction functions that live on the entire triangle instead of just correcting terms on the edges. 

Motivated by fluid applications, we naturally seek methods for higher-order finite element spaces, and in this context edge based modifications are not enough. Guided by our recent work in the piecewise linear case and also by other papers; see for example \cite{MR2740492,MR2377272}, it appears to us that the construction and addition to the right-hand side of a correction function is the key to achieve a higher-order method. This also appears in the finite difference context, for example a fourth order method  was developed by Marques, Nave and Rosales \cite{MR2823566} using a correction function approach. Very recently this idea was materialized by Adjerid, Ben-Romdhane and Lin \cite{MR3218337}. They developed  higher-order methods for problems involving discontinuous coefficients (which is a more general problem). However, they use strongly the assumption that the interface is a straight line. The key is to use both, the jump condition and the PDE, to find higher-order jump conditions. Inspired by their results, we define correction functions for any polynomial of degree $k$ for curved interfaces.

The contribution of this paper is in the direction of \cite{MR3218337}, we develop a higher-order piecewise continuous finite element method for problem \eqref{Problem} with
curved interfaces. Specifically, in Section \ref{SectionFEM} we develop notation and propose a finite element method, for each polynomial of degree $k$, introducing  a correction function $w_T^u$ and adding it to the variational formulation, only modifying the right-hand side of the equation. To do this, we construct this correction function incorporating the jump conditions of the exact solution on the interface. To the best of our knowledge, this is the first family of numerical methods (with any order approximation) were one can prove optimal accuracy for the above problem without modifying the stiffness matrix. Besides the novel method, an important contribution of this paper is the techniques used to give a-priori error estimation
analysis for the Poisson and the Stokes interface problems. In particular, an interpolation estimate using the correction function is given(see Lemma \ref{wapprox}), which is crucial for the full analysis of the proposed method. In Section \ref{SectionErrorAnalysis} we prove that our method is $k+1$ (mod a logarithmic factor) order accurate  in the maximum norm, if piecewise polynomials of degree $k$ are used.

As mentioned before, in this paper we also consider a finite element approximation to a Stokes interface problem, i.e., under the same geometry assumptions for problem \eqref{Problem}, we seek for a velocity vector $\n{u}$ and pressure $p$ satisfying
\begin{subequations}\label{StokesProblem}
\begin{align}
                             -\Delta \n{u}+\nabla p  &= \n{f}     \qquad\mbox{in } \Omega, \label{StokesProblem:a} \\
                                   \nabla\cdot \n{u} &= \,0         \qquad\mbox{in } \Omega, \label{StokesProblem:b} \\
                                                   \n{u} &= \,0         \qquad\mbox{on } \partial \Omega, \label{StokesProblem:c} \\
  \left[ D_{\n{n}}\n{u} -  p \n{n} \right] &= \n{\beta} \qquad\mbox{on } \Gamma, \label{StokesProblem:d}
\end{align}
\end{subequations}

As is well known, time dependent versions of problem \eqref{StokesProblem} have many applications in biology, see for example \cite{Cortez}. In fact, the immersed boundary method is used primarily for solving time dependent versions of \eqref{StokesProblem} with possibly nonlinear terms; see for example Peskin and Tu \cite{MR1185651}. Subsequently, LeVeque and Li \cite{MR1443639} consider this problem applying immersed interface method techniques. Some extension and finite element versions of these approaches can be found in \cite{MR2864640,MR2046114,MR1848735}.
We consider the numerical method in this paper for problem \eqref{StokesProblem} as an important step toward defining higher-order methods for time-dependent problems with moving interfaces.  We will show in Section \ref{SectionStokes} that the same methodology used to achieve higher-order methods for Poisson problem \eqref{Problem} can be applied to Stokes problem \eqref{StokesProblem}, achieving optimal convergence results.

In Section \ref{SectionNumericalExamples} we test the methods introduced in Section \ref{SectionFEM} and \ref{SectionStokes} with some numerical examples that illustrate the properties proven in previous sections. Furthermore, we provide in the Appendix quadrature formulas for the integration over curved regions crucial to achieve higher-order results, although our analysis considers exact integration.


\vskip4mm
\section{Finite element methods}\label{SectionFEM}
In this section we present a finite element method for problem \eqref{Problem} using continuous piecewise polynomials of degree $k$. We assume that the data $\beta$ is smooth. Furthermore, we assume that $u^{\pm}  \in C^{k+1}({\overline{\Omega}}^{\pm})$ and $f|_{\Omega^{\pm}}\equiv f^{\pm} \in C^{k-1}({\overline{\Omega}}^{\pm})$.

\subsection{Notation}

Let $\mathcal{T}_h$, $0 < h <1$ be a sequence of triangulations of $\Omega$, $ \overline{\Omega} = \cup_{T\in \mathcal{T}_h} \overline{T}$, with the elements $T$ mutually disjoint. We assume the mesh is shape regular, see \cite{MR1278258}. We adopt the convention that edges, elements, regions are  open sets, and  we use
the overline symbol to refer to their closure.  Let $ h_T$ denote the diameter of the
element $T$ and $h = \max_{T} h_T$. Let $V_h$ be the space of  continuous, piecewise polynomials of degree $k$, i.e.,
\begin{equation*}\label{Vh}
  V_h = \{ v\in  C(\Omega)\cap H_0^1(\Omega) \,:\, v|_T \in \mathbb{P}^k(T) \quad \forall T\in \mathcal{T}_h\},
\end{equation*}
where $\mathbb{P}^k(T)$ is the space of polynomial of degree less than or equal to $k$ on $T$ and $H_0^1(\Omega)$ the space of functions in $H^1(\Omega)$ vanishing at $\partial \Omega$.  

Next, we define an interpolant onto $V_h$.
\begin{Def}\label{DefinitionI_h}
Given  $v^{\pm} \in C(\overline{\Omega}^\pm)$, we define locally $I_h v \in V_h$
such that
\begin{equation}\label{interpolant}
  I_h v|_T(\theta)\,\,=\,\,\left\{
    \begin{array}{ll}
      v^{-}(\theta), & \hbox{if } \theta\in \overline{\Omega}^{-}  \\
      v^{+}(\theta), & \hbox{if } \theta \in \Omega^{+},
    \end{array}
  \right.
\end{equation}
for all $\theta\in T$, the degree $k$ Lagrange points of $T$.
\end{Def}
Note that if $v$ is continuous $I_h v$ is simply the Lagrange interpolant of $v$. However, if $v$ is discontinuous then $I_h v$ interpolates values of $v$ on Lagrange points not intersecting  $\Gamma$ and for Lagrange points lying on $\Gamma$ it takes the values of $v$ coming from $\Omega^{-}$ (this is without loss of generality). The following proposition states the stability result of the interpolant $I_h$.
\begin{proposition}
Let $v^{\pm} \in C(\overline{\Omega}^\pm)$ and $I_h$ defined above, then we have
\begin{equation}
\|I_h v\|_{L^\infty(T)}\,\,\leq\,\,C \|v\|_{L^\infty(T)}\quad \forall T\in \mathcal{T}_h.
\end{equation}
\end{proposition}

To prove  $L^2$ based error estimates  shape regularity of the meshes suffices.  However, as is well known, some form of quasi-uniformity of the mesh is needed to prove max-norm estimates, even if there is no interface \cite{Demlow}. Since we will prove max-norm estimates of our method,  and in order to avoid unnecessary details, we will assume that the meshes are quasi-uniform.

Also, for the sake of making the presentation simpler to the reader, we make the following assumption.
\begin{hyp} \label{hypho}
We assume here that the interface $\Gamma$ intersects the boundary of each triangle $T \in \mathcal{T}_h$ at most at two points. If $\Gamma$ intersects the boundary of a triangle $T$ in exactly two points, then these two points must be on different edges $\bar{e}$ of $T$.
\end{hyp}

Next, let $\mathcal{T}_h^{\Gamma}$ denote the set of triangles $T \in \mathcal{T}_h$ such that
$T$ intersects $\Gamma$, that is,  $T \cap \Gamma \not= \emptyset$. For each $T \in \mathcal{T}_h^{\Gamma}$ let $y_T$ and $z_T$ be the two endpoints of  $T\cap \Gamma$, see figure~\ref{fig:Fig1}. Let $L_T$ denote the line segment connecting $y_T$ and $z_T$.  Let $\n{\eta}$ be the unit vector perpendicular to $L_T$ and pointing outward $T^{-}$. Let also $\n{\tau}$ be the unit vector parallel to the line  $L_T$ such that $\n{\tau}$ is the rotation of
$\n{\eta}$ by ninety degrees counterclockwise.

For each $\ell=0, 1, \ldots, k,$ let $\{\bar{x}_i^{\ell,T}\}_{i=0}^\ell$ denote the Gauss points of the segment $L_T$. For each $\bar{x}_i^{\ell,T}$, let $Q_i^{\ell,T}$ be the line perpendicular to the line segment $L_T$ that passes through the points $\bar{x}_i^{\ell,T}$. We then define $x_i^{\ell,T}=Q_i^{\ell,T} \cap \Gamma$, for $i=0,\,\ldots,\,\ell$. Note here that the choice of Gauss points is a preference of the authors, related to the quadrature rules, but not essential in the proofs. We could also use, for instance, equally spaced points.

\begin{figure}[H]
\centering
\includegraphics[scale = .4]{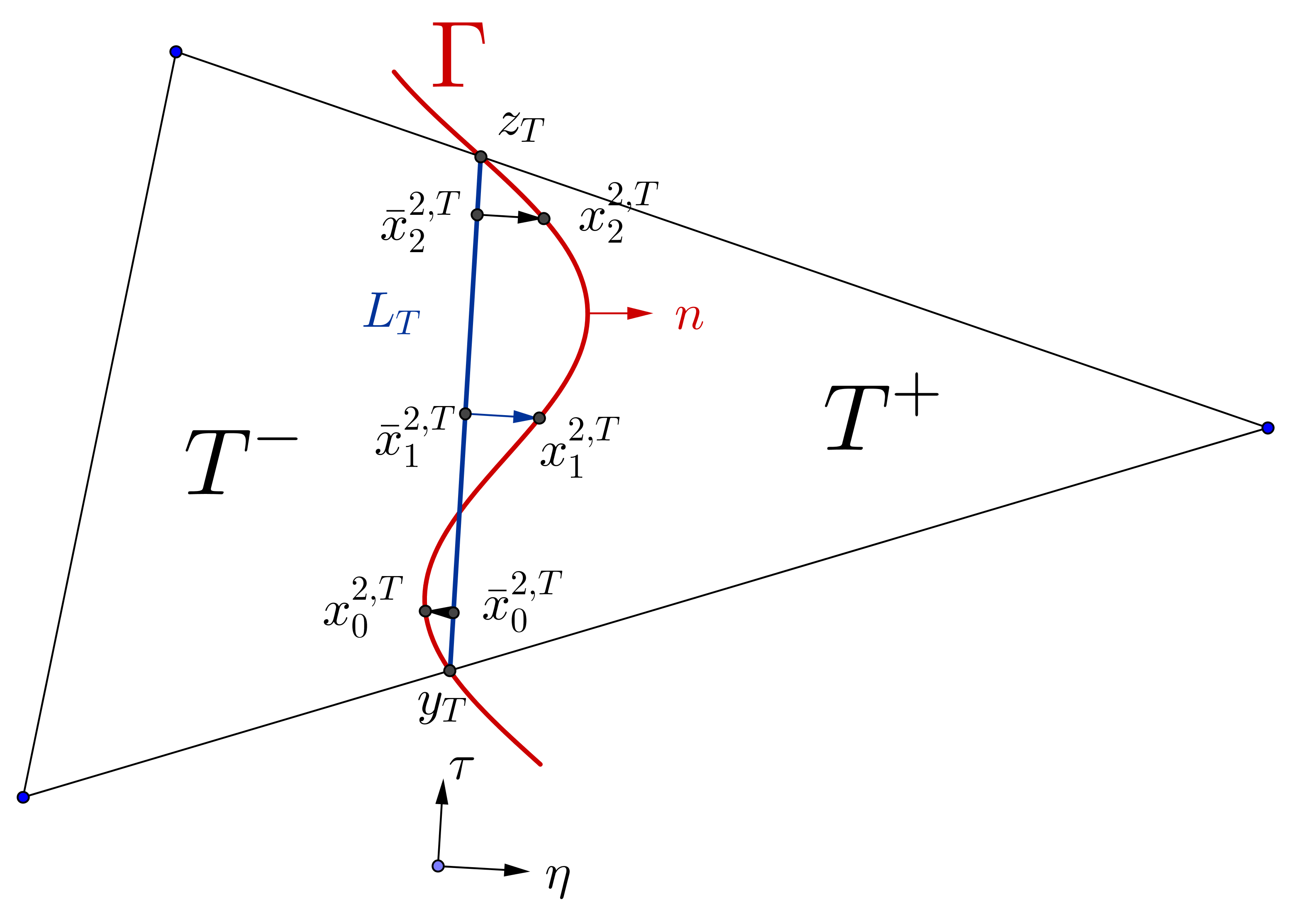} \\
\caption{Illustration of our notation, for $T\in \mathcal{T}_h^\Gamma$.}\label{fig:Fig1}
\end{figure}

Letting $T^{\pm}:= T \cap \Omega^{\pm}$, we define the following space for $T\in \mathcal{T}_h^\Gamma$
\begin{equation}\label{Sk}
S^{k}(T)=\left\{w \in L^2(T): w|_{T^{\pm}} \in \mathbb{P}^k(T^{\pm}) \right\}.
\end{equation}

\subsection{The proposed finite element method}
We now present our finite element method for problem \eqref{Problem}. Find $u_h \in V_h$, such that
\begin{equation}\label{fem}
\int_{\Omega} \nabla u_h \cdot \nabla v \, dx= \int_{\Omega} f v \, dx + \int_{\Gamma} \beta v \, ds- \sum_{T \in  \mathcal{T}_h^{\Gamma}}  \int_T \nabla w_T^u \cdot \nabla v \, dx
\end{equation}
for all $v \in V_h$, where $w_T^u$ is a correction function to be defined in Section \ref{correctionfn}.

\subsection{The correction function}\label{correctionfn}
We now show how to construct a piecewise polynomial function $w_T^u \in S^{k}(T)$ that will help to correct the right-hand side of the natural finite element method (\eqref{fem} without the correction term) to render it higher-order. We note that the functions $w_T^u$, for $T \in \mathcal{T}_h^\Gamma$,
are discontinuous across elements and satisfy jump conditions at Gauss
points on $\Gamma \cap T$.  Suppose that we give you a function $u$, then we let $w_T^u$ be the unique function in $S^k(T)$ (see Lemma \ref{lemmajump}) that satisfies

\begin{alignat}{1}
\left[ D_{\n{\eta}}^{k-\ell} w_T^u (x_i^{\ell,T})\right]=& \left[ D_{\n{\eta}}^{k-\ell} u(x_i^{\ell,T})\right] \quad \text{ for } 0 \le i \le \ell \text{ and } 0 \le  \ell \le k, \label{correction1}\\
w_T^u(\theta)= &\,0 \quad \text{ for  all degree $k$ Lagrange points } \theta \text{ of } T. \label{correction2}
\end{alignat}

We would like to stress that if $u$ is the solution to \eqref{Problem}
then we know all the jumps of $u$ in terms of the data $\beta$ and $f$
(see Section \ref{jump}) so we can construct $w_T^u$ a priori. It is
important to notice that we impose directional derivatives jump
conditions in  the ${\n{\eta}}$-direction rather than in the
${\n{n}}$-direction.  A reason for this choice is because we can show unisolvence, when $\Gamma$  is curved,
for any polynomial of degree $k$. Additionally, the construction of the correction functions can be made explicitly, that is, it
does not require solving any local linear system. Such construction permits us to develop not only
an elegant analysis for any polynomial degree $k$ but also an analysis
without imposing strong condition on how $\Gamma$ intersects an element $T$ in order to
control the ill conditioning of this matrix. We next show how
to transform jump data described in  ${\n{n}}$-direction to jump data in the ${\n{\eta}}$-direction.

\subsection{Data of jumps}\label{jump}
In this section we show that we can obtain the jumps of higher derivatives of $u$ across $\Gamma$ from data, $\beta$ and $f$. Similar ideas were used in \cite{MR1286215, MR2823566, MR2244270}.

Fix a point $x \in \Gamma$. Let $\n{n}$ be the normal vector to $\Gamma$ at $x$, and $\n{t}$ the tangent vector to $\Gamma$ at $x$.  We clearly have  $\left[D_{\n{n}} u(x)\right]= \beta$ and $\left[D_{\n{t}} u(x)\right]=0$. In fact, we also have $\left[D_{\n{t}}^\ell u(x)\right]=0$ for any $\ell$. Note that from the Poisson's equation \eqref{Problem:a},  we have
\begin{equation*}
-D_{\n{n}}^2 u- D_{\n{t}}^2 u = f \quad \mbox{in } \Omega^\pm,
\end{equation*}
and so $\left[D_{\n{n}}^2 u (x)\right]=\left[f(x)\right]$. Moreover, we note that $\left[D_{\n{t}} D_{\n{n}} u \right(x)]=D_{\n{t}} \left[ D_{\n{n}} u (x)\right]=\beta'(x)$ on $\Gamma$.

Now we proceed by induction. Suppose we have all jumps of the derivatives of order $\ell-1$ in terms of $\beta$ and $f$. Then, we will show how to get the jumps of derivatives of order $\ell$. Let $i+j=\ell$. If $i \ge 1$ then
\begin{equation*}
\left[D_{\n{t}}^i D_{\n{n}}^j u(x)\right] = D_{\n{t}} \left[D_{\n{t}}^{i-1} D_{\n{n}}^j u(x)\right]
\end{equation*}

Otherwise, using Laplace's equation we have
\begin{equation*}
\left[D_{\n{n}}^\ell u(x)\right]=\left[D_{\n{n}}^{\ell-2} f(x)\right] - D_{\n{t}} \left[D_{\n{t}} D_{\n{n}}^{\ell-2} u(x)\right].
\end{equation*}

Suppose that we would like the jump of $u$ in a different direction, say $\n{\eta}$. Then, we write $\n{\eta} = a \n{n}+b \n{t}$ obtaining
\begin{equation*}
\left[D_{\n{\eta}}^\ell u(x)\right]=\sum_{j=0}^\ell \binom{l} {j} a^{j} b^{\ell-j} \left[D_{\n{n}}^j D_{\n{t}}^{\ell-j} u(x)\right].
\end{equation*}

\vskip4mm
\section{Error Analysis}\label{SectionErrorAnalysis}
The objective of this section is to prove rigorous pointwise error estimates for the above method, which we achieve in Theorems \ref{maxgraderrorthm} and \ref{maxerrorthm}. Before doing this we need some technical results associated to
the subspace $S^k(T)$ and approximation properties of the correction function $w_T^u$.

\subsection{Properties of $S^k(T)$}
We now introduce
some crucial lemmas related to the space  $S^k(T)$.

\begin{lemma}\label{crucial}
Let $T\in \mathcal{T}_h^\Gamma$ and define $r_T=|L_T|$. For any $v \in \mathbb{P}^k(T)$, we have
\begin{equation*}
h_T^j\|D^j v\|_{L^\infty(T)} \le C \, h_T^{k} \sum_{\ell=0}^k \frac{1}{r_T^{\ell}} \max_{ 0 \le i \le \ell} |D_{\n{\eta}}^{k-\ell} v (x_i^{\ell,T})|, \quad j=0,1,
\end{equation*}
where the constant $C$ depends only on the shape regularity of $T$,  the polynomial degree $k$ and the regularity of $\Gamma$.
\end{lemma}
\begin{proof}
Let $L_T^E$ be the segment with length $|L_T^E|=2 h_T$, centered at $(y_T+z_T)/2$ and aligned with $L_T$. This guaranties that any point $x \in T$ can be projected orthogonally on  $L_T^E$.
Then, using Taylor's expansion on $T$ from $L_T^E$, we can easily show that
\begin{equation*}
h_T^j\|D^j v\|_{L^\infty(T)} \le C \sum_{\ell=0}^k  h_T^ {k-\ell} \|D_{\n{\eta}}^{k-\ell} v \|_{L^\infty(L_T^E)}.
\end{equation*}

Using Taylor's theorem on $L_T^E$ from a point of  $L_T$ we get
\begin{equation*}
\|D_{\n{\eta}}^{k-\ell} v \|_{L^\infty(L_T^E)} \le C \sum_{s=0}^{\ell} h_T^s \|D_{\n{\tau}}^s D_{\n{\eta}}^{k-\ell} v \|_{L^\infty(L_T)}, \quad j=0,1.
\end{equation*}

The inverse inequality gives
\begin{equation*}
\|D_{\n{\tau}}^s D_{\n{\eta}}^{k-\ell} v \|_{L^\infty(L_T)} \le \frac{C}{r_T^s}  \|D_{\n{\eta}}^{k-\ell} v \|_{L^\infty(L_T)}.
\end{equation*}

We therefore have
\begin{align*}
h_T^j\|D^j v\|_{L^\infty(T)} \,&\le\,  C h_T^k\sum_{\ell=0}^k  h_T^ {-\ell} \sum_{s=0}^\ell \frac{h^s_T}{r_T^s} \|D_{\n{\eta}}^{k-\ell} v \|_{L^\infty(L_T)}\\
&=\,  C h_T^k\sum_{\ell=0}^k  \sum_{s=0}^\ell \frac{1}{h_T^{\ell-s}}\frac{1}{r_T^s} \|D_{\n{\eta}}^{k-\ell} v \|_{L^\infty(L_T)}\\
&\le\,  C h_T^k\sum_{\ell=0}^k  \sum_{s=0}^\ell \frac{1}{r_T^{\ell-s}}\frac{1}{r_T^s} \|D_{\n{\eta}}^{k-\ell} v \|_{L^\infty(L_T)}\\
& =\, C \, h_T^{k} \sum_{\ell=0}^k \frac{1}{r_T^{\ell}}  \|D_{\n{\eta}}^{k-\ell} v \|_{L^\infty(L_T)},
\end{align*}
where we used that $r_T \le h_T$.

To bound  the right-hand side above, we use induction on $\ell$. First, using that $D_{\n{\eta}}^{k} v$ is a constant we have
\begin{equation*}
\|D_{\n{\eta}}^{k} v \|_{L^\infty(L_T)}= |D_{\n{\eta}}^{k}v(x_0^{0,T})|.
\end{equation*}

Assume that we have proved
\begin{equation}\label{hyp}
\sum_{\ell=m+1}^k \frac{1}{r_T^{\ell}}  \|D_{\n{\eta}}^{k-\ell} v \|_{L^\infty(L_T)} \le C\, \sum_{\ell=m+1}^k \frac{1}{r_T^{\ell}} \max_{ 0 \le i \le \ell} |D_{\n{\eta}}^{k-\ell} v (x_i^{\ell,T})|,
\end{equation}
then we want to prove that
\begin{equation}\label{conclusion}
\sum_{\ell=m}^k \frac{1}{r_T^{\ell}}  \|D_{\n{\eta}}^{k-\ell} v \|_{L^\infty(L_T)} \le C\, \sum_{\ell=m}^k \frac{1}{r_T^{\ell}} \max_{ 0 \le i \le \ell} |D_{\n{\eta}}^{k-\ell} v (x_i^{\ell,T})|.
\end{equation}

Since $D_{\n{\eta}}^{k-m} v$  is a polynomial of degree $m$, we have that
\begin{equation*}
\|D_{\n{\eta}}^{k-m} v \|_{L^\infty(L_T)} \le C \, \max_{ 0 \le i \le m} |D_{\n{\eta}}^{k-m} v (\bar{x}_i^{m,T})|.
\end{equation*}

Using Taylor's theorem we have
\begin{equation*}
|D_{\n{\eta}}^{k-m} v (\bar{x}_i^{m,T})| \le |D_{\n{\eta}}^{k-m} v (x_i^{m,T})|+ \sum_{\ell=m+1}^k  d^{\ell-m} \|D_{\n{\eta}}^{k-\ell} v \|_{L^\infty(L_T)},
\end{equation*}
where $\displaystyle d\,=\,\max_{ 0 \le i \le \ell \le k} |x_{i}^{\ell,T}-\bar{x}_{i}^{\ell,T}|$.  It is clear that $d \le C r_T^2$, since we have assumed that
$\Gamma$ is smooth, that is the radius of curvature is $O(1)$, however
we note that we will only use that $d \le C r_T$. Hence, we have
\begin{equation*}
\frac{1}{r_T^m} \max_{0\le i \le m} |D_{\n{\eta}}^{k-m} v (\bar{x}_i^{m,T})| \le \frac{1}{r_T^m}  \max_{0 \le i \le m} |D_{\n{\eta}}^{k-m} v (x_i^{m,T})|+ C \sum_{\ell=m+1}^k  \frac{1}{r_T^{2m-\ell}} \|D_{\n{\eta}}^{k-\ell} v \|_{L^\infty(L_T)}.
\end{equation*}

However, $\frac{1}{r_T^{2m-\ell}} \le \frac{1}{r_T^{\ell}}$ for $\ell \ge m+1$ and so using  \eqref{hyp} we arrive at \eqref{conclusion}.
\end{proof}

The following is a fundamental lemma for the construction of $w_T^u$ and the
estimation of $h_T^j \|D^j w_T^u\|_{L^\infty(T^\pm)}$ when we choose
$c_{i,\ell} = \left[ D_{\n{\eta}}^{k-\ell} u(x_i^{\ell,T})\right]$, see equation \eqref{correction1}.

\begin{lemma}\label{lemmajump}
Given data $\{c_{i,\ell}\}$ for $0\le i \le \ell$ and $0 \le  \ell \le k$. There exists a unique function $w \in S^k(T)$, such that
\begin{alignat}{1}
\label{wsystem1}\left[ D_{\n{\eta}}^{k-\ell} w (x_i^{\ell,T})\right]=&\,\,c_{i,\ell} \quad \text{ for } 0 \le i \le \ell \text{ and } 0 \le  \ell \le k, \\
\label{wsystem2}w(\theta)=&\,\,0 \quad \text{ for  all the degree $k$ Lagrange points } \theta \text{ of } T,
\end{alignat}
with the following bound
\begin{equation}\label{lemmajump3}
h_T^j \|D^j w\|_{L^\infty(T^\pm)} \le C h_T^k \sum_{\ell=0}^k \frac{1}{r_T^{\ell}} \max_{ 0 \le i \le \ell} |c_{i,\ell}|, ~~~~\mbox{for} ~~j=0,1,
\end{equation}
where $C$ depends only on the shape regularity of $\mathcal{T}_h$, the polynomial degree $k$ and the regularity of $\Gamma$.
\end{lemma}

\begin{proof}
We will construct $w$ of the form $w=z-I_h z$, where $z \in S^k(T)$. Notice that, by definition of $I_h$, $z-I_h z$ vanishes on the $(k+1)(k+2)/2$ Lagrange points of $T$, satisfying \eqref{wsystem2}. Moreover, since $I_h z$ is smooth on $T$
\begin{alignat*}{1}
\left[ D_{\n{\eta}}^{k-\ell} w (x_i^{\ell,T})\right]= \left[  D_{\n{\eta}}^{k-\ell} z (x_i^{\ell,T}) \right].
\end{alignat*}

The function $z$ will be given by
\[
z=\begin{cases} 0 &\mbox{ in } T^{+}, \\
v & \mbox{ in } T^{-}, \end{cases}
\]
where $v$ is the unique polynomial on $\mathbb{P}^k$, such that
\begin{equation*}
D_{\n{\eta}^-}^{k-\ell} v (x_i^{\ell,T})=\,\,c_{i,\ell} \quad \text{ for } 0 \le i \le \ell \text{ and } 0 \le  \ell \le k.
\end{equation*}

The existence and uniqueness of $v$ follow from representing $v = v(s,r)$ as a polynomial of
degree $k$ in  $r$ (${\n{\tau}}$-direction) and $s$ (${\n{\eta}}$-direction), where $s=0$
represents the straight line passing through $y_T$ and $z_T$,  then by
decomposing
\[
v(s,r) = p_k(r) + s p_{k-1}(r) + s^2 p_{k-2}(r)+ \cdots +s^k p_0,
\]
where $p_\ell$, for $ 0\leq \ell\leq k$, is a polynomial of degree $\ell$ in $r$.
It is easy to see, by using interpolation at the Gauss point $\bar{x}_0^{0,T}$,
that $p_0$ exists and is unique, then by using interpolation at the Gauss points
$\bar{x}_1^{0,T}$ and $\bar{x}_1^{1,T}$ that $p_1(r)$ exists and is unique, and so on.

According to Lemma \ref{crucial}, we have the following bound
\begin{equation}\label{vinq}
h_T^j\|D^j v\|_{L^\infty(T)} \le C \,  h_T^k \sum_{\ell=0}^k \frac{1}{r_T^{\ell}} \max_{ 0 \le i \le \ell} |c_{i,\ell}|.
\end{equation}

Hence, $w$ satisfies \eqref{wsystem1}. Moreover,
\begin{equation*}
h_T^j \|D^j w\|_{L^\infty(T^\pm)} =h_T^j \|D^j (z-I_h z)\|_{L^\infty(T^\pm)} \le h_T^j \|D^j v\|_{L^\infty(T^-)}+ h_T^j \|D^j I_h z\|_{L^\infty(T)}.
\end{equation*}

Using an inverse estimate and stability of the interpolant, we have
\begin{equation*}
 h_T^j \|D^j I_h z\|_{L^\infty(T)} \le C   \|I_h z\|_{L^\infty(T)} \le  C \, \|z\|_{L^\infty(T)} \le C\, \|v\|_{L^\infty(T^-)}.
\end{equation*}

Hence,
\begin{equation*}
h_T^j \|D^j w\|_{L^\infty(T^\pm)} \le h_T^j \|D^j v\|_{L^\infty(T^-)}+  \|v\|_{L^\infty(T^-)}.
\end{equation*}

We get \eqref{lemmajump3} once we apply \eqref{vinq}.
\end{proof}

\subsection{Approximation properties of $w_T^u$.}

Since we are assuming $u^{\pm} \in C^{k+1}(\Omega^{\pm})$ and $\Gamma$
is smooth, there exist extensions $u_E^{\pm} \in C^{k+1}(\Omega)$ (see
Lemma 6.27 \cite{MR1814364}), such that the following holds
\begin{alignat*}{1}
u_E^{\pm}\,=\,&u^{\pm} \quad \text{ on } \Omega^{\pm},\qquad
\|u_E^{\pm}\|_{C^{k+1}(\Omega)} \, \le\,  C\, \|u^{\pm}\|_{C^{k+1}(\Omega^{\pm})}.
\end{alignat*}

Let $B_{2r_T} \subset T^E$ be a ball of radius $2r_T$ that encloses
$T^\Gamma$.  Here $T^E$  is the smallest patch of triangles of the mesh
$\mathcal{T}_h$ on the neighborhood of $T\in \mathcal{T}_h^\Gamma$. Let $J_T$ be the $L^2$ projection onto  polynomials of degree $k$ in $B_{2r_T}$ and consider its natural extension to all of $T^E$. Then, we can prove the following lemma.

\begin{lemma}  Let $w\in C^{k+1}(T^E)$, then we have
\begin{equation}\label{Jr}
r_T^j \|D^j(J_T(w)-w)\|_{L^\infty(B_{2 r_T})} \le C \, r_T^{k+1} \|w\|_{C^{k+1}(B_{2 r_T})} \quad  \text{ for } 0 \le j \le k,
\end{equation}
and
\begin{equation}\label{Jh}
h_T^j \|D^j(J_T(w)-w)\|_{L^\infty(T^E)} \le C \, h_T^{k+1} \|w\|_{C^{k+1}(T^E)}  \quad  \text{ for } 0 \le j \le k.
\end{equation}
\end{lemma}
\begin{proof}
The inequality \eqref{Jr} is a standard approximation of the $L^2$ projection. To prove \eqref{Jh} we apply Taylor's theorem to get
\begin{equation*}
h_T^j \|D^j(J_T(w)-w)\|_{L^\infty(T^E)} \le h_T^j  \sum_{\ell=1}^{k-j} h_T^\ell \|D^{j+\ell}(J_T(w)-w)\|_{L^\infty(B_{2 r_T})}+ h_T^{k+1} \|w\|_{C^{k+1}(T^E)}.
\end{equation*}
The result follows after applying \eqref{Jr} and using that $r_T \le h_T$.
\end{proof}

The following lemma establishes the approximation result for the correction function defined in Section \ref{correctionfn}.

\begin{lemma}\label{wapprox}
Suppose the solution  $u$ to problem \eqref{Problem} satisfies $u^{\pm}  \in C^{k+1}(\Omega^{\pm})$, and $w_T^u$ is the correction function defined by \eqref{wsystem1} and \eqref{wsystem2}. Then, we have
\begin{equation*}
h_T^j \|D^j(u-I_hu-w_T^u)\|_{L^\infty(T^\pm)} \le C \,  h_T^{k+1} \left(\|u_E^+\|_{C^{k+1}(T^E)} + \|u_E^-\|_{C^{k+1}(T^E)}\right),~~~\mbox{for} ~~j=0,1
\end{equation*}
and $C$ depends only on the shape regularity of $T$, the polynomial degree $k$
and the regularity of $\Gamma$.
\end{lemma}

\begin{proof}
We will define $ v\in S^k(T)$ (see \eqref{Sk}) as follows
\[
v=\begin{cases} J_T(u_E^+) &\mbox{ on } T^{+}, \\
J_T(u_E^-) & \mbox{ on } T^{-}. \end{cases}
\]

Clearly $v-I_h v=w_T^v$ by Lemma \ref{lemmajump}. Hence, we have
\begin{alignat*}{1}
h_T^j \|D^j(u-I_h u-w_T^u)\|_{L^\infty(T^\pm)} \,=\,\,&h_T^j \|D^j\left((u-v) +(v-I_h v)+I_h(v-u)-w_T^u \right)\|_{L^\infty(T^\pm)}\\
                                  \,=\,\,&h_T^j \|D^j\left((u-v)-I_h(u-v)-w_T^{u-v}\right)\|_{L^\infty(T^\pm)} \\
                                  \,\le\,\,& C\, h_T^j \left(\|D^j(u-v)\|_{L^\infty(T^\pm)} + \|D^j w_T^{u-v}\|_{L^\infty(T^\pm)}\right) \\
                                  \,\,\,& + C \,\|u-v\|_{L^\infty(T)},
\end{alignat*}
where we used $w_T^{u-v} = w_T^{u} - w_T^{v}$, an inverse estimate and the stability of $I_h$ in the max-norm.

Using \eqref{Jh} we get
\begin{equation*}
h_T^j \|D^j(u-v)\|_{L^\infty(T^\pm)}+ \,\|u-v\|_{L^\infty(T^\pm)} \le C h_T^{k+1} (\|u_E^-\|_{C^{k+1}(T^E)}+ \|u_E^+\|_{C^{k+1}(T^E)}).
\end{equation*}

Estimate \eqref{lemmajump3} implies
\begin{equation*}
h_T^j \|w_T^{u-v}\|_{L^\infty(T)} \le  C  \, h_T^k \sum_{\ell=0}^k \frac{1}{r_T^{k-\ell}} \|\left[D_{\n{n}}^\ell (v-u) \right]\|_{L^\infty(T \cap \Gamma)}.
\end{equation*}

Applying \eqref{Jr} we obtain
\begin{equation*}
\sum_{\ell=0}^k \frac{1}{r^{k-\ell}} \|\left[D_{\n{n}}^\ell (v-u) \right]\|_{L^\infty(T \cap \Gamma)} \le C \, r_T (\|u_E^+\|_{C^{k+1}(B_{2r_T})} + \|u_E^-\|_{C^{k+1}(B_{2r_T})}),
\end{equation*}
which completes the proof.
\end{proof}

\subsection{Error estimates}
The next lemma will show that the correction term $w_T^u$ in the finite element method \eqref{fem} will allow us to compare $I_h u-u_h$.
\begin{lemma}\label{lemmaIuuh} Let $u^\pm\in C^{k+1}(\Omega^\pm)$ be the solution of \eqref{Problem} and $u_h$ be the solution of \eqref{fem}. Then, it holds
\begin{equation*}
 \int_{\Omega} \nabla (I_h u-u_h) \cdot \nabla v \, dx \le  C \, h^{k} \|\nabla v\|_{L^1(\Omega)} \left(\|u^+\|_{C^{k+1}(\Omega^+)}+\|u^-\|_{C^{k+1}(\Omega^-)}\right)\quad
\forall v\in V_h,
\end{equation*}
where $C$ depends only on the shape regularity of $\{\mathcal{T}_h\}_{h>0}$, the polynomial degree $k$ and the regularity of $\Gamma$.
\end{lemma}

\begin{proof}
\begin{alignat*}{1}
\int_{\Omega} \nabla (I_h u-u_h) \cdot \nabla v \, dx=& \int_{\Omega} \nabla (I_h u-u) \cdot \nabla v \, dx+  \int_{\Omega} \nabla u \cdot \nabla v \, dx-\int_{\Omega} \nabla  u_h \cdot \nabla v \, dx \\
=&  \int_{\Omega} \nabla (I_h u-u) \cdot \nabla v \, dx + \sum_{T \in  \mathcal{T}_h^{\Gamma}}  \int_T \nabla w_T^u \cdot \nabla v \, dx  \\
=& \sum_{T \in \mathcal{T}_h \backslash  \mathcal{T}_h^{\Gamma}} \int_{T} \nabla(I_h u-u)\cdot \nabla v \, dx +\sum_{T \in \mathcal{T}_h^{\Gamma}} \int_{T} \nabla(u-I_h u-w_T^u)\cdot\nabla v \, dx
\end{alignat*}

The result now easily follows from Lemma \ref{wapprox} and the fact that $u$ is smooth on $T \in \mathcal{T}_h \backslash  \mathcal{T}_h^{\Gamma}$.
\end{proof}

From the above lemma we can easily prove an optimal estimate in the $H^1$ semi-norm:
\begin{equation*}
\|\nabla(I_h u-u_h)\|_{L^2(\Omega)} \le C \, h^{k} \left(\|u^+\|_{C^{k+1}(\Omega^+)}+\|u^-\|_{C^{k+1}(\Omega^-)}\right).
\end{equation*}

However, are goal is to prove estimates in the maximum-norm as our next result states. A slightly sub-optimal (off by a log factor) can be proved if we use the above lemma directly. In order to prove the optimal estimate, we will give a more involved argument.
\begin{theorem}\label{maxgraderrorthm}
Suppose that $\Omega$ is convex. Let $u^\pm\in C^{k+1}(\Omega^\pm)$ be the solution of \eqref{Problem} and $u_h$ be the solution of \eqref{fem}, then
\begin{equation*}
\|\nabla(I_h u-u_h)\|_{L^\infty(\Omega)} \le C \, h^{k} (\|u^+\|_{C^{k+1}(\Omega^+)}+\|u^-\|_{C^{k+1}(\Omega^-)}),
\end{equation*}
where $C$ depends only on the shape regularity of $\{\mathcal{T}_h\}_{h>0}$,
the polynomial degree $k$ and the regularity of $\Gamma$.
\end{theorem}

\begin{proof}
Let $e_h=I_h u-u_h$ and suppose that the maximum of $|\partial_{x_i} e_h|$ occurs at $z \in \Omega$ (for some fixed $1 \le i \le 2$). Suppose $z \in T$, for some $T \in \mathcal{T}_h$.  Consider now the regularized Dirac delta function $\delta_h^{{z}} = \delta_h  \in C^1_0(T_{{z}})$ (see \cite{MR1278258}), which satisfies

\begin{equation}\label{rz}
  r({z}) = (r,\delta_h)_{T_{{z}}},\qquad \forall r \in P^k(T_{{z}}),
\end{equation}
and has the following property

\begin{equation}\label{deltaestimates}
  \|\delta_h\|_{W^{r,q}(T_{{z}})}\,\,\leq\,\,C h^{-r-2(1-1/q)},\qquad 1\leq q\leq\infty,\,\,r=0,1.
\end{equation}

For each $i=1,2$, define the approximate Green's function $g \in H_0^1(\Omega)$, which solves the following equation:

\begin{subequations}\label{Greensfp}
\begin{align}
 -\Delta g &= \partial_{x_i}\delta_h\qquad\mbox{in }\Omega, \label{Greensfp:a} \\
  g&= 0\qquad \quad\,\,\,\mbox{on }\partial\Omega. \label{Greensfp:b}
\end{align}
\end{subequations}

We also consider its finite element approximation $g_h \in V_h$ that satisfies
\begin{equation}\label{approxgreen}
\int_{\Omega} \nabla g_{h} \cdot \nabla v\, d{x}= \int_{\Omega} v \,\partial_{x_i} \delta_h \,d{x}  \quad \text{ for all } v \in V_h.
\end{equation}

From the work of Scott and Rannacher \cite{MR645661} we  have
\begin{equation}\label{gmgh}
\|\nabla(g-g_h)\|_{L^1(\Omega)} \le C.
\end{equation}

Moreover, using a dyadic decomposition one can show
\begin{equation}\label{gglobal}
\|\nabla g\|_{L^1(\Omega)} \le C \log(1/h).
\end{equation}

A log free estimate holds if we consider a smaller domain, i.e.
\begin{equation}\label{glocal}
\|\nabla g\|_{L^1(S^\Gamma)} \le C,
\end{equation}
where $S^{\Gamma}=\{x \in \Omega: dist(x, \Gamma) \le \kappa h\} $ for some fixed constant $\kappa$; see for instance \cite{GSS2014}.
Hence, combining \eqref{gmgh} and \eqref{glocal} we have
\begin{equation}\label{ghlocal}
\|\nabla g_h\|_{L^1(S^\Gamma)} \le C.
\end{equation}

We start by using the definition of $\delta_h$ and problem \eqref{approxgreen}
\begin{equation*}
\| \partial_{x_i} e_h\|_{L^\infty(\Omega)}=|\partial_{x_i} e_h(z)|=|\int_{\Omega} \delta_h \partial_{x_i} e_h \, dx|=|\int_{\Omega} \partial_{x_i} \delta_h\,  e_h \, dx|.
\end{equation*}

Then, we see that
\begin{equation*}
\| \partial_{x_i} e_h\|_{L^\infty(\Omega)}= |\int_{\Omega} \nabla g \cdot \nabla e_h \, dx|= |\int_{\Omega} \nabla g_h \cdot \nabla e_h \, dx|.
\end{equation*}

If we follow the proof of Lemma \ref{lemmaIuuh} we see that
\begin{equation*}
\int_{\Omega} \nabla g_h \cdot \nabla e_h \, dx= J_1+J_2,
\end{equation*}
where
\begin{equation*}
 J_1\,=\,\sum_{T \in \mathcal{T}_h \backslash  \mathcal{T}_h^{\Gamma}} \int_{T} \nabla(I_h u-u) \cdot\nabla g_h \, dx,
\end{equation*}
and
\begin{equation*}
J_2\,=\,\sum_{T \in \mathcal{T}_h^{\Gamma}} \int_{T} \nabla(I_h u+w_T^u-u)\cdot \nabla g_h \, dx.
\end{equation*}

Applying Cauchy-Schwarz inequality to $J_2$ we get
\begin{equation*}
 J_2 \,\le\, C \,\|\nabla g_h\|_{L^1(S^\Gamma)} \max_{ T \in \mathcal{T}_h^{\Gamma}}  \|\nabla(I_h u +w_T^u-u)\|_{L^\infty(T)}.
\end{equation*}

Moreover, using \eqref{ghlocal} and Lemma \ref{wapprox} we have
\begin{equation*}
 J_2 \le C  \, h^{k} (\|u^+\|_{C^{k+1}(\Omega^+)}+\|u^-\|_{C^{k+1}(\Omega^-)}).
\end{equation*}

To give an estimate for $J_1$ we define $R_h = \cup_{T \in \mathcal{T}_h \backslash  \mathcal{T}_h^{\Gamma}} T$. Now, adding and subtracting $\nabla g$, we obtain
\begin{equation*}
J_1=\sum_{T \in \mathcal{T}_h \backslash  \mathcal{T}_h^{\Gamma}} \int_{T} \nabla (I_h u-u) \cdot\nabla  g_h \, dx= \int_{R_h} \left( \nabla (I_h u-u) \cdot\nabla  (g_h-g) +  \nabla (I_h u-u) \cdot\nabla  g \, \right)dx .
\end{equation*}

Using \eqref{gmgh}, we have
\begin{equation*}
\int_{R_h} \nabla (I_h u-u)\cdot \nabla  (g_h-g) \, dx \,\le\, C \, h^{k} (\|u^+\|_{C^{k+1}(\Omega^+)}+\|u^-\|_{C^{k+1}(\Omega^-)}).
\end{equation*}

For the remaining term we integrate by parts to get
\begin{equation*}
\int_{R_h}  \nabla (I_h u-u)\cdot \nabla  g \, dx= \int_{R_h}  (I_h u-u) \partial_{x_i} \delta_h \, dx+ \int_{\partial R_h \backslash \partial \Omega} (I_h u-u) D_{\n{n}}  g \, ds.
\end{equation*}

Here we used that, since $\Omega$ is convex, $\nabla g$ is continuous and so integration by parts makes sense.

Clearly we have
\begin{equation*}
\int_{R_h}  (I_h u-u)  \partial_{x_i} \delta_h  \, dx \le C h^{k} \left(\|u^+\|_{C^{k+1}(\Omega^+)}+\|u^-\|_{C^{k+1}(\Omega^-)}\right).
\end{equation*}

Finally, we have
\begin{alignat*}{1}
\int_{\partial R_h \backslash \partial \Omega} (I_h u-u) D_{\n{n}}  g \, ds \,\le\, & C \, \|I_h u-u\|_{L^\infty(R_h)} \| D_{\n{n}}  g\|_{L^1(\partial R_h \backslash \partial \Omega)} \\
 \,\le\, &  C h^{k+1}  \| D_{\n{n}}  g\|_{L^1(\partial R_h \backslash \partial \Omega)}  \left(\|u^+\|_{C^{k+1}(\Omega^+)}+\|u^-\|_{C^{k+1}(\Omega^-)}\right).
\end{alignat*}

In Appendix \ref{AppendixDngbound} we prove the bound
\begin{equation}\label{Dngbound}
\| D_{\n{n}}  g\|_{L^1(\partial R_h \backslash \partial \Omega)}  \le \frac{C}{h},
\end{equation}
which will then show that
\begin{equation*}
J_1 \le  C h^{k} (\|u^+\|_{C^{k+1}(\Omega^+)}+\|u^-\|_{C^{k+1}(\Omega^-)}),
\end{equation*}
and will complete the proof.
\end{proof}

Next, we will prove an estimate for the error in the maximum norm. In the case there is no interface a logarithmic factor is not present for $k \ge 2$ (see \cite{MR1278258}), however, we do not see how to remove this factor in our setting.
\begin{theorem}\label{maxerrorthm}
Suppose that $\Omega$ is convex. Let $u^\pm\in C^{k+1}(\Omega^\pm)$ be the solution of \eqref{Problem} and $u_h$ be the solution of \eqref{fem}, then
\begin{equation}
\|I_h u - u_h\|_{L^\infty(\Omega)} \leq C h^{k+1} \log(1/h) \left(\|u^+\|_{C^{k+1}(\Omega^+)}+\|u^-\|_{C^{k+1}(\Omega^-)}\right).
\end{equation}
where $C$ depends only on the shape regularity of $\{\mathcal{T}_h\}_{h>0}$, the polynomial degree $k$ and the regularity of $\Gamma$.
\end{theorem}
\begin{proof}
We follow the proof of Theorem 5 in \cite{GSS2014}. Let $z\in \Omega$ be arbitrary and let $\delta_h = \delta_h^z$ defined in proof of Theorem \ref{maxgraderrorthm}. Let $\tilde{g}$ satisfy

\begin{alignat}{2}\label{gtildeeqn}
 -\Delta \tilde{g} &= \delta_h \qquad && \mbox{in }\Omega,  \\
  \tilde{g}&= 0\qquad && \mbox{on }\partial\Omega,
\end{alignat}
and consider its continuous piecewise linear finite element approximation $\tilde{g}_h$. Then
\begin{align*}
(I_h u-u_h)({z})=&\,\,\int_{\Omega} (I_h u-u_h) \delta_h\, d{x}\,\,=\,\, \int_{\Omega} \nabla (I_h u-u_h)\cdot \nabla \tilde{g}_h \, d{x} \\
          = & \sum_{T\in \mathcal{T}_h\backslash \mathcal{T}_h^\Gamma} \int_T \nabla(I_hu-u) \cdot \nabla \tilde{g}_h)dx +  \sum_{T\in \mathcal{T}_h^\Gamma}\int_T \nabla(I_hu+w_T^u-u)\cdot\nabla  \tilde{g}_h)dx \\
            &=:\quad J_1 \,+ \,J_2.
\end{align*}

We first give a bound for $J_2$
\begin{align*}
J_2 &\,\leq\, C \max_{T\in \mathcal{T}_h^\Gamma} \| I_hu+w_T^{u}-u\|_{L^\infty(T)} \|\nabla \tilde{g}_h\|_{L^1(T)}  \\
&\,\leq\, C h^k \left(\|u^+\|_{C^{k+1}(\Omega^+)}+\|u^-\|_{C^{k+1}(\Omega^-)}\right)\|\nabla \tilde{g}_h\|_{L^1(S^\Gamma)} \\
&\,\leq \,C h^k \left(\|u^+\|_{C^{k+1}(\Omega^+)}+\|u^-\|_{C^{k+1}(\Omega^-)}\right)\left(\|\nabla \tilde{g}_h-\nabla \tilde{g}\|_{L^1(S^\Gamma)} +\|\nabla \tilde{g}\|_{L^1(S^\Gamma)}\right).
\end{align*}
In \cite{GSS2014} we proved $\|\nabla \tilde{g}\|_{L^1(S^\Gamma)}\leq C h \log(1/h) $, therefore

\begin{align*}
J_2 &\,\leq\ C h^{k+1}\log(1/h) \left(\|u^+\|_{C^{k+1}(\Omega^+)}+\|u^-\|_{C^{k+1}(\Omega^-)}\right).
\end{align*}

Now for $J_1$, we will use the Raviart-Thomas projection (see \cite{MR0483555}) $\Pi : H^1(\Omega) \rightarrow \Phi_h^D$, defined locally for any $T\in\mathcal{T}_h$, $\Pi|_T : H^1(T) \rightarrow RT_0(T)$, where
\[
RT_0(T) = [\mathbb{P}^0(T)]^2 \oplus x \mathbb{P}^0(T),\quad \Phi_h^D = \left\{ \n{\phi}\in [L^2(\Omega)]^2:\,\,\n{\phi}|_{T} \in RT_0(T)\,\,\forall T\in \mathcal{T}_h \right\}.
\]

Then, we observe
\[
J_1 = J_1(\nabla \tilde{g}_h)  = J_1(\nabla \tilde{g}_h-\Pi \nabla \tilde{g})+ J_1(\Pi \nabla \tilde{g}).
\]

In the Appendix of \cite{GSS2014} we prove the estimate $ \|\nabla \tilde{g}_h-\Pi \nabla \tilde{g}\|_{L^1(\Omega)} \leq h\log{1/h}$, then we clearly we have
\begin{align*}
J_1(\nabla \tilde{g}_h-\Pi \nabla \tilde{g})&\,\leq\, C h^{k} \left(\|u^+\|_{C^{k+1}(\Omega^+)}+\|u^-\|_{C^{k+1}(\Omega^-)}\right) \|\nabla \tilde{g}_h-\Pi \nabla \tilde{g}\|_{L^1(\Omega)} \\
& \,\leq \,C h^{k+1}\log(1/h) \left(\|u^+\|_{C^{k+1}(\Omega^+)}+\|u^-\|_{C^{k+1}(\Omega^-)}\right).
\end{align*}

Using that $\Pi (\nabla \tilde{g})$ is piecewise constant and has continuous normal components across edges, we have after integration by parts
\begin{equation*}
J_1(\Pi \nabla \tilde{g}) = \sum_{e \in \mathcal{E}_h^{\Gamma, \partial} } \int_{e} (I_h u-u) \Pi \nabla \tilde{g} \cdot n,
\end{equation*}
where $\mathcal{E}_h^{\Gamma, \partial}$  are set of edges that are both an edge of a triangle in $ \mathcal{T}_h\backslash \mathcal{T}_h^\Gamma$ and a triangle in $\mathcal{T}_h\backslash \mathcal{T}_h^\Gamma$.

Therefore, we see that
\begin{alignat*}{1}
J_1(\Pi \nabla \tilde{g}) \,=\, & \sum_{e \in \mathcal{E}_h^{\Gamma, \partial} } \int_{e} (I_h u-u) \Pi \nabla \tilde{g} \cdot n \\
\,\le\, & C h^{k}  \, \|\Pi \nabla \tilde{g}\|_{L^1(S^\Gamma)}  \left(\|u^+\|_{C^{k+1}(\Omega^+)}+\|u^-\|_{C^{k+1}(\Omega^-)}\right)  \\
\,\le\, &  C \log(1/h) h^{k+1}  \, \left(\|u^+\|_{C^{k+1}(\Omega^+)}+\|u^-\|_{C^{k+1}(\Omega^-)}\right),
\end{alignat*}
where again we used $\|\Pi \nabla \tilde{g}\|_{L^1(S^\Gamma)}\leq C h \log(1/h)$, which follows from results in  \cite{GSS2014}.

\end{proof}

\begin{rem} \label{rem1}Since $u - (u_h+w_T^u) = (I_h u - u_h) +
(u - w_T^u - I_hu)$, by using triangle inequality
and Lemma \ref{wapprox}, $u_h+w_T^u$ approximates  $u$ on $T$
by the same estimates given in
Theorems \ref{maxgraderrorthm} and \ref{maxerrorthm}.
\end{rem}

\vskip4mm
\section{Stokes Interface Problem}\label{SectionStokes}
In this section we consider the Stokes interface problem in two dimensions introduced in equation \eqref{StokesProblem}.
Equivalently, we can incorporate the jump condition, equation \eqref{StokesProblem:d}, as follows
\begin{subequations}\label{StokesProblem2}
\begin{align}
  -\Delta \n{u}+\nabla p  &= \n{f}+\n{B} \qquad\mbox{in }\Omega \label{StokesProblem2:a} \\
        \nabla\cdot \n{u} &= 0           \qquad \mbox{in }  \Omega \label{StokesProblem2:b} \\
                        \n{u} &= 0           \qquad \mbox{on }  \partial \Omega \label{StokesProblem2:c}
\end{align}
\end{subequations}

where
\begin{equation*}
\n{B}(x) = \int_{0}^A \n{\beta}(s) \delta(x-\n{X}(s)) ds,
\end{equation*}
and $ \n{X}(s)$ with $0<s<A$ is the arc-length parametrization of the interface $\Gamma$.


Following the ideas of LeVeque and Li in \cite{MR1443639}, we can easily write individual jump conditions for the velocity and the pressure in terms of the tangential and normal component of the data $\n{\beta}$. Let $\theta$ be the angle between the $x$-direction ($x$-axis) and $\n{n}$-direction pointing outward the interface $\Gamma$ at a point $\n{X}(s)$. Then, we write the normal and tangential components

\begin{equation*}
\hat{\n{\beta}}(s)= \left(
  \begin{array}{cc}
    \hat{\beta}_1 \\
     \hat{\beta}_2 \\
  \end{array} \right)  = \left(
  \begin{array}{cc}
    \cos(\theta) & \sin(\theta) \\
    -\sin(\theta) & \cos(\theta) \\
  \end{array} \right) \n{\beta}(s).
\end{equation*}

The jumps conditions for the velocity and the pressure are given by
\begin{align}
\nonumber[p] &=  \hat{\beta}_1,  \\
\label{JumpsStokes} \left[D_{\n{n}} p \right] &= \frac{d}{ds}\hat{\beta}_2,\\
\nonumber\left[\n{u}\right] & = 0,\\
\nonumber\left[D_{\n{n}} \n{u} \right] & = \n{\beta} + \hat{\beta}_1 \n{n}.
\end{align}

For the sake of completeness, we present the derivation of this jumps in Appendix \ref{Ajumps}.

\subsection{The finite element method.}

We first present the standard variational formulation of Stokes interface problem \eqref{StokesProblem}. Find $(\n{u}, p) \in [H_0^1(\Omega)]^2 \times L_0^2(\Omega)$, such that

\begin{align}\label{ContinuousProblem}
\int_{\Omega} \nabla \n{u} : \nabla \n{v} \,d{x}  - \int_{\Omega} p\,\nabla\cdot\n{v} d{x} & = \int_{\Omega} \n{f} \cdot\n{v} d{x} +\int_{\Gamma}  \n{\beta}\cdot \n{v} ds  \qquad \forall \n{v}\in [H_0^1(\Omega)]^2,\\
\nonumber \int_{\Omega} q\,\nabla\cdot\n{u} d{x} & = 0\qquad \forall q\in L_0^2(\Omega),
\end{align}
where $L_0^2(\Omega)\,=\,\{ q\in L^2(\Omega):\,\,\int_\Omega q = 0 \}$.

As before, we consider a sequence of triangulations of $\bar{\Omega}$, $\mathcal{T}_h$ with $0<h<1$, and $ \overline{\Omega} = \cup_{T\in \mathcal{T}_h} \overline{T}$, with the elements $T$ mutually disjoint. Let $ h_T$ denote the diameter of the element $T$ and $h = \max_{T} h_T$. We assume that the mesh is quasi-uniform and shape regular.

We consider  a class of finite element subspaces $\n{V}_h\subset [H_0^1(\Omega)]^2$ and $M_h \subset L_0^2(\Omega)$ satisfying the following assumptions:
\begin{enumerate}[label=\bfseries A\arabic*]
  \item \label{A1}  $\n{V}_h$ and $M_h $ are a pair of inf-sup stables subspaces, with $\n{V}_h \subset [H_0^1(\Omega)]^2$.

  \item \label{A2} We let $k \ge 1$ as the maximum integer such that
  \begin{align*}
  \n{V}^k_h&\,:=\,\left\{\n{v}\in  C(\Omega)\cap[H_0^1(\Omega)]^2:\,\n{v}|_T\in [\mathbb{P}^{k}(T)]^2,\,\,\forall T\in \mathcal{T}_h\right\}\,\subseteq\, \n{V}_h,
  \end{align*}
  and, if $M_h$ contains the discontinuous pressure space of degree $k-1$ we let
  \begin{align*}
  M^{k-1}_h &\,:=\,\left\{q \in   L_0^2(\Omega) :\,q|_T\in \mathbb{P}^{k-1}(T),\,\,\forall T\in \mathcal{T}_h\right\}\,\subseteq\, M_h,
  \end{align*}
  otherwise
  \begin{align*}
  M^{k-1}_h &\,:=\,\left\{q \in  C(\Omega)\cap L_0^2(\Omega) :\,q|_T\in \mathbb{P}^{k-1}(T),\,\,\forall T\in \mathcal{T}_h\right\}\,\subseteq\, M_h.
  \end{align*}
\end{enumerate}

For instance, $k=1$ for the pair $\mathbb{P}_2^2 - \mathbb{P}_0$, reduced $\mathbb{P}_2^2 - \mathbb{P}_0$
and mini element, while $k=2$ for Taylor-Hood  $\mathbb{P}_2^2 - \mathbb{P}_1$.

We next define the interpolant onto these spaces. We let $I_h$ be the interpolant defined componentwise in \eqref{interpolant} onto the space $\n{V}^k_h$, and $J_h$ be the interpolant defined in \eqref{interpolant} onto the space $M^{k-1}_h$, in the case of continuous pressure finite element spaces. Otherwise, if $M_h$ contains discontinuous finite element pressure spaces we define $J_h$ to be the $L^2$ projection onto $M^{k-1}_h$.

Find  $(\n{u}_h,p_h) \in  \n{V}_h\times M_h$, such that
\begin{align}\label{Stokesfem}
\nonumber \int_{\Omega} \nabla \n{u}_h : \nabla \n{v} \,d{x}  - \int_{\Omega} p_h\nabla\cdot\n{v} d{x} \,\,= & \int_{\Omega} \n{f} \cdot\n{v} d{x} +\int_{\Gamma}  \n{\beta}\cdot \n{v} ds  \\
                                                                                         & -\sum_{T\in \mathcal{T}_h^\Gamma}\left(\int_{T} \nabla \n{w}_T^{\n{u}} :\nabla\n{v} \, d{x} +
\int_{T} w_T^p\nabla\cdot\n{v} d{x}    \right),  \\
\nonumber \int_{\Omega} q\nabla\cdot\n{u}_h d{x} \,\,= &  -\sum_{T\in \mathcal{T}_h^\Gamma} \int_{T} q \nabla\cdot \n{w}_T^{\n{u}} d{x},
\end{align}
for all $(\n{v},q) \in  \n{V}_h\times M_h$.

Using the last two equations of \eqref{JumpsStokes}, the correction function $\n{w}_T^{\n{u}}$ is defined
componentwise as in Section \ref{correctionfn}.
The same is true for $w_T^p$ when $J_h$ is the Lagrange interpolant by using the first two equations
of \eqref{correctionfn}. In the case $J_h$ is the $L^2$ projection we replace equation \eqref{correction2} with the condition $J_h (w_T^p) =0$, i.e.
\begin{alignat}{1}
\left[ D_{\n{\eta}}^{k-\ell} w_T^p (x_i^{\ell,T})\right]=& \left[ D_{\n{\eta}}^{k-\ell} p(x_i^{\ell,T})\right] \quad \text{ for } 0 \le i \le \ell \text{ and } 0 \le  \ell \le k, \label{correctionp1}\\
J_h(w_T^p)= &\,0. \label{correctionp2}
\end{alignat}

Note that each component of  $\n{w}_T^{\n{u}}$ and  $w_T^p$ are obtained
independently to each other, therefore, the Lemma \ref{wapprox} can be applied to each one separately.

Similarly to Lemma \ref{wapprox}, we have the following estimate for the interpolant $J_h$ and the correction function $w_T^p$
\begin{equation*}
\|p-J_hp-w_T^p\|_{L^\infty(T^\pm)} \le C \,  h_T^{k+1} \left(\|p^+\|_{C^{k+1}(T)} + \|p^-\|_{C^{k+1}(T)}\right)\quad\forall T\in \mathcal{T}_h^\Gamma,
\end{equation*}
where $C$ is a constant depending on the shape regularity of $T$, the constant $k$ and the regularity of $\Gamma$.  Then, the following result holds and can be proved similar to Lemma \ref{lemmaIuuh}.

\begin{lemma}\label{lemmaStokes} Let $(\n{u},p)$ be solution of \eqref{StokesProblem} and assume that $\n{u}^\pm\in[C^{k+1}(\Omega^\pm)]^2$ and $p^\pm\in C^{k}(\Omega^\pm)$. Let $\n{V}_h$ and $M_h$ be the finite element spaces satisfying assumptions {\bf A1-A2} and consider the definitions above for $k$, $I_h$ and $ J_h$. Let $(\n{u}_h, p_h)\in \n{V}_h\times M_h$  be solution of \eqref{Stokesfem}. Then, it holds
\begin{alignat*}{1}
 \int_{\Omega} \nabla (I_h \n{u}- &\n{u}_h ): \nabla \n{v} \,d{x}  \,-\, \int_{\Omega} (J_h p- p_h)\nabla\cdot\n{v} d{x} \,\,  \\ &  \le \,\, C \, h^{k} \|\nabla \n{v}\|_{L^1(\Omega)}\left(\|\n{u}^+\|_{C^{k+1}(\Omega^+)}+\|\n{u}^-\|_{C^{k+1}(\Omega^-)}+\|p^+\|_{C^{k}(\Omega^+)}+\|p^-\|_{C^{k}(\Omega^-)}\right)  \\
 &\\
 \int_{\Omega} q \nabla\cdot(& \n{u}_h -I_h \n{u})d{x} \,\,\le  \,\,  C h^{k} \|q\|_{L^1(\Omega)} \left(\|\n{u}^+\|_{C^{k+1}(\Omega^+)}+\|\n{u}^-\|_{C^{k+1}(\Omega^-)}\right),
\end{alignat*}
\end{lemma}
for all  $(\n{v},\,q) \,\in\, \n{V}_h\times M_h$.

\begin{proof}
For the first equation we use \eqref{ContinuousProblem}  and \eqref{Stokesfem} to obtain
\begin{align*}
\int_{\Omega} \nabla (I_h\n{u}-\n{u}_h) : \nabla \n{v} \,&d{x} - \int_{\Omega} (J_h p- p_h)\nabla\cdot\n{v} d{x}  \,\, \\
= & \int_{\Omega} \nabla (I_h\n{u}-\n{u}) : \nabla \n{v} \,d{x} -   \int_{\Omega} (J_h p- p)\nabla\cdot\n{v} d{x}  \\
&\,\,+\,\,\sum_{T\in \mathcal{T}_h^\Gamma}\left(\int_{T} w_T^p\nabla\cdot\n{v} d{x}   + \int_{T} \nabla \n{w}_T^{\n{u}} :\nabla\n{v} \, d{x} \right) \\
=&\sum_{T\in \mathcal{T}_h\backslash\mathcal{T}_h^\Gamma} \left(\int_{T} \nabla (I_h\n{u}-\n{u}) : \nabla \n{v} \,d{x} -   \int_{T} (J_h p- p)\nabla\cdot\n{v} d{x} \right)  \\
&\,\ +\,\,
\sum_{T\in \mathcal{T}_h^\Gamma}\left(\int_{T} \nabla(I_h\n{u}-\n{u}+ \n{w}_T^{\n{u}}) :\nabla\n{v} \, d{x} +
\int_{T} (J_h p- p+w_T^p)\nabla\cdot\n{v} d{x}\right).
\end{align*}

Similarly for the second equation we have
\begin{align*}
\int_{\Omega} q \nabla\cdot( I_h \n{u}-\n{u}_h )d{x} \,\,&=  \,\, \int_{\Omega} q \nabla\cdot( I_h \n{u}-\n{u} )d{x} +  \sum_{T\in \mathcal{T}_h^\Gamma} \int_{T} q \nabla\cdot \n{w}_T^{\n{u}} d{x} \\
& =\sum_{T\in \mathcal{T}_h\backslash\mathcal{T}_h^\Gamma} \int_{T} q \nabla\cdot (I_h \n{u}-\n{u}) d{x}+\sum_{T\in \mathcal{T}_h^\Gamma} \int_{T} q \nabla\cdot (I_h \n{u}-\n{u}+\n{w}_T^{\n{u}}) d{x}.
\end{align*}

Then the results follow from the properties of the correction functions $\n{w}_T^{\n{u}}$ and $w_T^p$ Lemma \ref{wapprox}, and the fact that $(\n{u},\,p)$ is smooth on $T\in\mathcal{T}_h\backslash \mathcal{T}_h^\Gamma$.
\end{proof}

Analogously to the proof of Theorem \ref{maxgraderrorthm}, we can prove the following result using approximate Green's function estimates for the Stokes problem. We give a sketch of the proof.

\begin{theorem}\label{thmStokes}
Let $(\n{u},p)$ be solution of \eqref{StokesProblem} and assume that $\n{u}^\pm\in[C^{k+1}(\Omega^\pm)]^2$ and $p^\pm\in C^{k}(\Omega^\pm)$. Let $\n{V}_h$ and $M_h$ be the finite element spaces satisfying assumptions {\bf A1-A2} and consider the definitions above for $k$, $I_h$ and $ J_h$. Let $(\n{u}_h, p_h)\in \n{V}_h\times M_h$  be solution of \eqref{Stokesfem}. Then, there exists a constant $C>0$, such that
\begin{align}\label{Stokeserror}
\|\nabla(I_h\n{u}-\n{u}_h)\|_{L^\infty(\Omega)} + \|J_h(p) - p_h\|_{L^\infty(\Omega)} \, \leq\, C h^k&\left(\|\n{u}^+\|_{C^{k+1}(\Omega^+)}+\|\n{u}^-\|_{C^{k+1}(\Omega^-)}\right. \\
\nonumber &\,\,\, \left.+\|p^+\|_{C^{k}(\Omega^+)}+\|p^-\|_{C^{k}(\Omega^-)}\right)
\end{align}
\end{theorem}
\begin{proof}
Let $\n{e}_h^{\n{u}}=I_h \n{u} - \n{u}_h$ and suppose that the maximum of $\partial_{x_j}(\n{e}_h)_{i} \,\,(i,j=1,2)$ occurs at $z\in \Omega$. Suppose $z\in T\in \mathcal{T}_h$. Considering the definition and properties of the regularized Dirac delta function $\delta_h = \delta_h^z$ introduced in \eqref{rz} and \eqref{deltaestimates}, we define the approximate Green's function $(\n{g},\,\lambda) \in [H_0^1(\Omega)]^2 \times L^2(\Omega)$ such that

\begin{align*}
-\Delta \n{g} + \nabla \lambda  &= (\partial x_j \delta_h)\n{\mathrm{e}}_i\quad \mbox{in } \Omega, \\
\nabla \cdot \n{g} &= 0 \quad \mbox{in } \Omega,\\
\n{g} &= 0\quad \mbox{on } \partial \Omega ,
\end{align*}
where $\n{\mathrm{e}}_i$ denotes the $i$th standard canonical basis vector of $\mathbb{R}^2$. We also consider its finite element approximation $(\n{g}_h,\,\lambda_h) \in \n{V}_h\times M_h$ that satisfies
\begin{align*}
\int_{\Omega} \nabla(\n{g} - \n{g}_h) : \nabla \n{\chi} d x - \int_{\Omega} (\lambda-\lambda_h)\nabla\cdot\n{\chi} dx & = 0 \quad \forall \n{\chi}\in V_h,\\
\int_{\Omega} \omega \nabla\cdot (\n{g}-\n{g}_h) dx &= 0 \quad \forall \omega\in M_h.
\end{align*}

Using the definition of $\delta_h$ we have

\begin{align*}
\|(\partial_{x_j} \n{e}_h^{\n{u}})_i\|_{L^\infty(\Omega)} \,=\, \left|(\partial_{x_j} \n{e}_h^{\n{u}}(z))_i\right|\,=\,\left| \int_{\Omega} \partial_{x_j}\delta_h e_i\cdot\n{e}_h^{\n{u}} d{x}\right|,
\end{align*}
and then

\begin{equation*}
\|(\partial_{x_j} \n{e}_h^{\n{u}})_i\|_{L^\infty(\Omega)} \,=\, \left| \int_{\Omega} \nabla\n{g} : \nabla \n{e}_h^{\n{u}} d x - \int_{\Omega} \lambda\nabla\cdot\n{e}_h^{\n{u}}d{x} \right| \,=\,\left| \int_{\Omega} \nabla\n{g}_h : \nabla \n{e}_h^{\n{u}}d x - \int_{\Omega} \lambda_h\nabla\cdot\n{e}_h^{\n{u}}d{x} \right|.
\end{equation*}

Following the proof of Lemma \ref{lemmaStokes}, we see that
\begin{equation*}
 \int_{\Omega} \nabla\n{g}_h : \nabla \n{e}_h^{\n{u}}d x - \int_{\Omega} \lambda_h\nabla\cdot\n{e}_h^{\n{u}}d{x}  \,= \, J_1\, +\, J_2,
\end{equation*}
where
\begin{align*}
J_1\,& =\,\sum_{T\in \mathcal{T}_h\backslash\mathcal{T}_h^\Gamma} \left( \int_{T} \nabla (I_h\n{u}-\n{u}) : \nabla \n{g}_h \,d{x}+ \int_{T} \lambda_h \nabla\cdot (I_h \n{u}-\n{u}) d{x} \right) \\
J_2\,& = \,\sum_{T\in \mathcal{T}_h^\Gamma}\left( \int_{T} \nabla(I_h\n{u}-\n{u}+ \n{w}_T^{\n{u}}) :\nabla\n{g}_h \, d{x}+\int_{T} \lambda_h \nabla\cdot (I_h \n{u}-\n{u}+\n{w}_T^{\n{u}}) d{x} \right).
\end{align*}
As in the proof of  Theorem \ref{maxgraderrorthm}, the estimates of $J_1$ and $J_2$ follows from the properties of the correction functions and bounds for the Green's functions that can be found for example in  \cite{MR2066358,MR2945141}. The estimate for the error of the pressure follows from similar arguments, we leave the details to the reader.
\end{proof}
\begin{rem} \label{rem2}
Using the same arguments as Remark \ref{rem1}, we can establish
the same a priori error estimates given in Theorem \ref{thmStokes}
for $\n{u} - (\n{u}_h +  \n{w}_T^{\n{u}})$ and
for  $ p - (p_h + w_T^p)$. We note that the method \eqref{Stokesfem}
gives $p_h \in M_h$, therefore, if $J_h$ is the $L^2$ projection, then
$ (p_h + w_T^p) \in L^2_0(\Omega)$, if  $J_h$ is the interpolation operator,
$ p_h + w_T^p$ might not have average zero and a constant $O(h^{k+1})$
must be added to make it in $L^2_0(\Omega)$.
\end{rem}

\section{Numerical Experiments}\label{SectionNumericalExamples}
We illustrate the performance of our method with some numerical examples. We consider the square domain $ \Omega = [-1,\,1]^2$, and we triangulate the domain with structured and non-structured triangular meshes. We tabulate the $L^2$ and $H^1$ semi-norm errors as well as the $L^\infty$ and $W^{1,\infty}$ semi-norm errors, with their respective order of convergence. Plots of approximate solutions and rate of convergence are also provided. For the case when the interface is not a straight line we need to integrate over curved region. We address this problem in Appendix \ref{quadrature} giving explicit quadrature formulas.

\subsection{Numerical examples for Poisson interface problem \eqref{Problem}.}

Let $u$ be the exact solution of problem \eqref{Problem}, $ u_h $ be the solution by the method defined in \eqref{fem}. We define the error with respect to the Lagrange interpolant $I_h$ and the respective order of convergence (associated to the error and the norm) as follows
\begin{equation*}
e_h := u_h - I_h u,\quad \mbox{r}(e,\|\cdot\|) := \frac{\log(\|e_{h_{l+1}}\|/\|e_{h_l}\|)}{ \log(h_{l+1}/h_{l})}.
\end{equation*}

We will illustrate our results with two numerical examples using piecewise quadratic polynomials, i.e., $k=2$. Note that $I_h$ is then the piecewise quadratic Lagrange interpolant.

\begin{enumerate}[label=\bfseries Ex. \theenumi]
\item  \label{itm:P1} Consider an exact solution of problem \eqref{Problem}
 \begin{equation*}
    u(x,y)\quad = \quad\left\{
       \begin{array}{ll}
         (2/3+x)^3, & \hbox{if } x < 1/3 \\
         (4/3-x)^3, & \hbox{if } x \geq 1/3.
       \end{array}
     \right.
\end{equation*}
In this case, the interface $\Gamma$ is a straight line. We summarize the results for structured and non-structured meshes in the following tables.

\begin{table}[H]
\centering
\begin{tabular}{|c|cc|cc|cc|cc|}\hline
$h$ & $\|e_h\|_{L^2}$ & $\mbox{r}$& $\|e_h\|_{L^\infty}$ & $\mbox{r}$ & $\|\nabla e_h\|_{L^2}$ & $\mbox{r}$& $\|\nabla e_h\|_{L^\infty}$ & $\mbox{r}$\\ \hline
3.5e-1 &    8.41e-5  &            &   1.83e-4   &           &   1.53e-3    &         &    3.72e-3&\\
1.8e-1 &    7.49e-6  &    3.49    &   2.33e-5   &    2.97   &   2.76e-4    &  2.47   &    9.29e-4   &    2.00\\
8.8e-2 &    6.63e-7  &    3.50    &   2.92e-6   &    2.99   &   4.92e-5    &  2.49   &    2.32e-4   &    2.00\\
4.4e-2 &    5.85e-8  &    3.50    &   3.64e-7   &    3.00   &   8.74e-6    &  2.49   &    5.80e-5   &    2.00\\
2.2e-2 &    5.16e-9  &    3.50    &   4.56e-8   &    3.00   &   1.55e-6    &  2.50   &    1.45e-5   &    2.00  \\\hline
\end{tabular}\vskip2mm
\caption{Errors and orders of convergence of method \eqref{fem} for \ref{itm:P1}, on structured meshes.}
\end{table}

\begin{table}[H]
\centering
\begin{tabular}{|c|cc|cc|cc|cc|}\hline
$h$ & $\|e_h\|_{L^2}$ & $\mbox{r}$& $\|e_h\|_{L^\infty}$ & $\mbox{r}$ & $\|\nabla e_h\|_{L^2}$ & $\mbox{r}$& $\|\nabla e_h\|_{L^\infty}$ & $\mbox{r}$\\ \hline
5.0e-1  &   4.92e-4  &           &    1.99e-3    &          &   9.95e-3    &          &   9.54e-2  & \\
2.5e-1  &   3.86e-5  &    3.67   &    2.37e-4    &   3.07   &   1.71e-3    &  2.54    &   2.41e-2     &  1.99\\
1.3e-1  &   3.86e-6  &    3.32   &    4.42e-5    &   2.43   &   3.57e-4    &  2.26    &   7.98e-3     &  1.59\\
6.3e-2  &   3.89e-7  &    3.31   &    5.99e-6    &   2.88   &   7.82e-5    &  2.19    &   1.92e-3     &  2.06\\
3.1e-2  &   4.44e-8  &    3.13   &    7.74e-7    &   2.95   &   1.87e-5    &  2.06    &   5.06e-4     &  1.92  \\\hline
\end{tabular}\vskip2mm
\caption{Errors and orders of convergence of method \eqref{fem} for \ref{itm:P1}, on non-structured meshes.}
\end{table}

We observe optimal order of convergence for both, structured and non-structured meshes, confirming the results of Theorem \ref{maxgraderrorthm} and \ref{maxerrorthm}. By structured mesh we mean bisecting in the
northeast direction each cell of a uniform rectangular grid. When $\Gamma$ is a straight line, the numerical
quadrature is simplified, otherwise we need to integrate over curved elements (see Appendix \ref{quadrature}). A superconvergence phenomenon is also observed in the $L^2$ norm for structured meshes, inherited from the problem without interface.
\vskip2mm

\item\label{itm:P2} Consider an exact solution of problem \eqref{Problem}
 \begin{equation*}
    u(x)\quad = \quad\left\{
       \begin{array}{ll}
         1, & \hbox{if } r \leq 1/3 \\
         1-\log(3r), & \hbox{if } r > 1/3
       \end{array}
     \right. \quad x\in [-1,1]^2,
\end{equation*}
 where $r =\|x\|_2$. We summarize the errors and order of convergence in the following tables
\begin{table}[H]
\centering
\begin{tabular}{|c|cc|cc|cc|cc|}\hline
$h$ & $\|e_h\|_{L^2}$ & $\mbox{r}$& $\|e_h\|_{L^\infty}$ & $\mbox{r}$ & $\|\nabla e_h\|_{L^2}$ & $\mbox{r}$& $\|\nabla e_h\|_{L^\infty}$ & $\mbox{r}$\\ \hline
7.1e-1 &    1.70e-2  &           &     4.76e-2 &            &     1.89e-1 &          &       4.37e-1 &\\
3.5e-1 &    1.73e-3  &    3.29   &    5.09e-3  &     3.22   &   3.66e-2   &   2.37   &    1.29e-1    &   1.76\\
1.8e-1 &    1.49e-4  &    3.54   &    7.41e-4  &     2.78   &   5.66e-3   &   2.69   &    3.05e-2    &   2.08\\
8.8e-2 &    1.22e-5  &    3.61   &    6.82e-5  &     3.44   &   8.48e-4   &   2.74   &    5.99e-3    &   2.34\\
4.4e-2 &    1.16e-6  &    3.40   &    1.39e-5  &     2.29   &   1.54e-4   &   2.46   &    1.78e-3    &   1.75\\
2.2e-2 &    1.09e-7  &    3.41   &    2.08e-6  &     2.74   &   2.71e-5   &   2.51   &    5.06e-4    &   1.81\\
1.1e-2 &    9.16e-9  &    3.57   &    2.39e-7  &     3.12   &   4.71e-6   &   2.52   &    1.36e-4    &   1.90\\ \hline
\end{tabular}\vskip2mm
\caption{Errors and orders of convergence of method \eqref{fem} for \ref{itm:P2}, on structured meshes.}
\end{table}

\begin{table}[H]
\centering
\begin{tabular}{|c|cc|cc|cc|cc|}\hline
$h$ & $\|e_h\|_{L^2}$ & $\mbox{r}$& $\|e_h\|_{L^\infty}$ & $\mbox{r}$ & $\|\nabla e_h\|_{L^2}$ & $\mbox{r}$& $\|\nabla e_h\|_{L^\infty}$ & $\mbox{r}$\\ \hline
1.8e-1   &  8.87e-5   &         &     3.97e-4   &          &    3.80e-3   &           &   2.53e-2   & \\
9.0e-2   &  9.73e-6   &   3.29  &     7.46e-5   &    2.49  &    9.04e-4   &   2.14    &   7.43e-3   &    1.82\\
4.7e-2   &  1.11e-6   &   3.33  &     1.06e-5   &    3.00  &    2.15e-4   &   2.21    &   2.58e-3   &    1.63\\
2.4e-2   &  1.30e-7   &   3.15  &     1.42e-6   &    2.95  &    5.06e-5   &   2.13    &   7.34e-4   &    1.84\\
1.2e-2   &  1.59e-8   &   3.14  &     2.24e-7   &    2.76  &    1.27e-5   &   2.07    &   2.16e-4   &    1.83\\
6.1e-3   &  1.96e-9   &   3.04  &     3.15e-8   &    2.85  &    3.15e-6   &   2.02    &   5.55e-5   &    1.98 \\\hline
\end{tabular}\vskip2mm
\caption{Errors and orders of convergence of method \eqref{fem} for \ref{itm:P2}, on non-structured meshes.}
\end{table}
\begin{figure}[H]
\includegraphics[scale=.45]{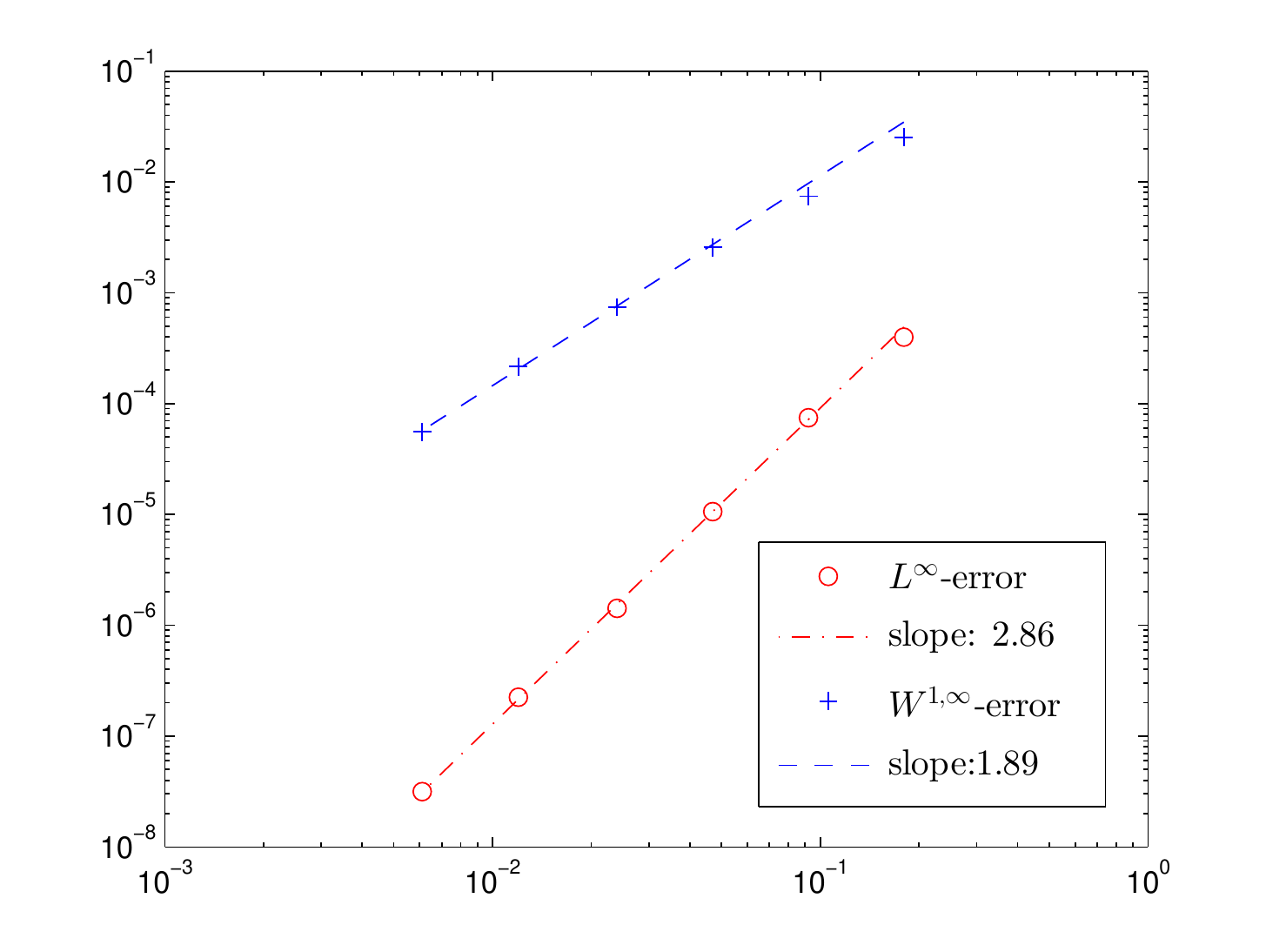} \,\,\,\includegraphics[scale=.35]{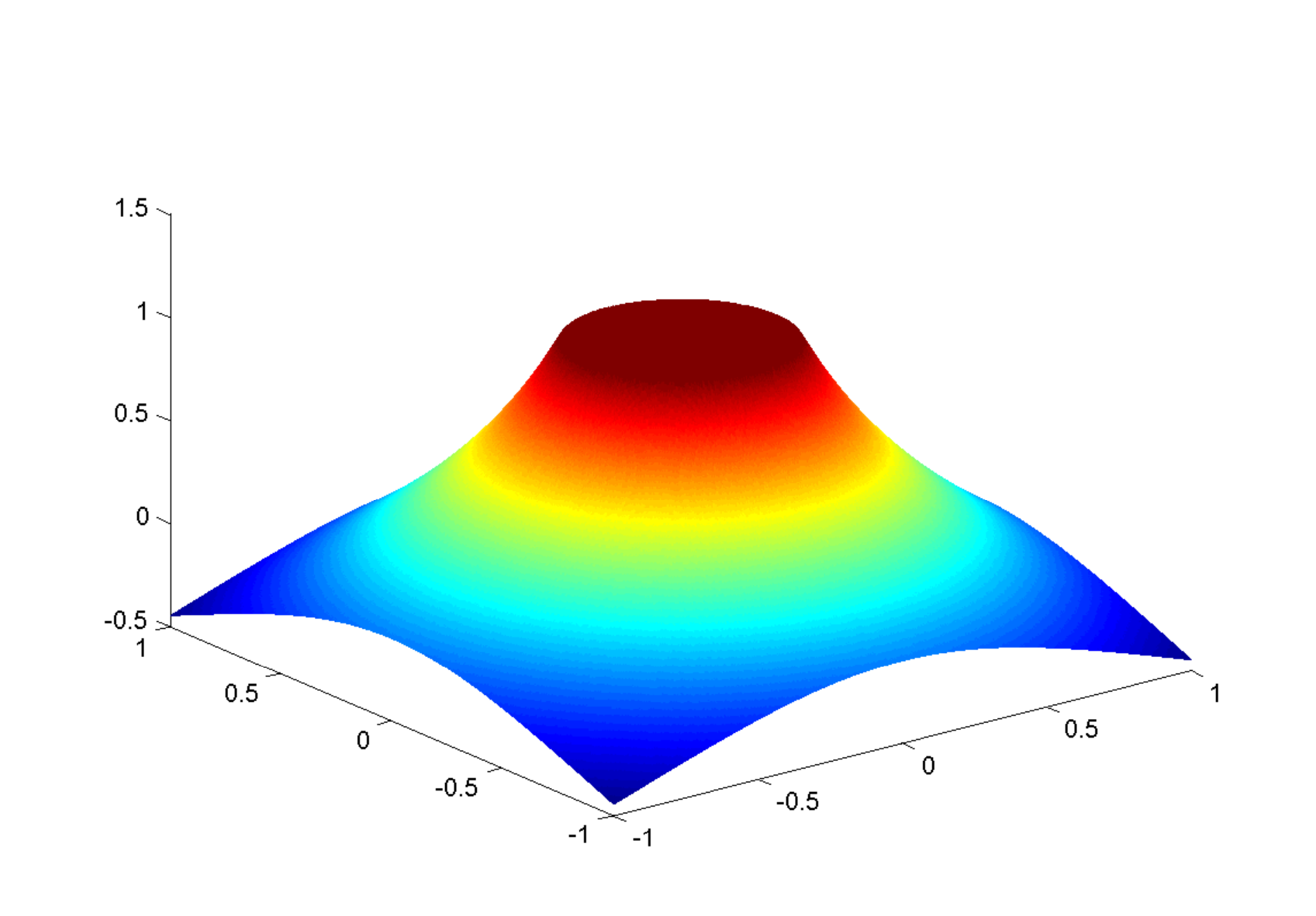} \\
\caption{Plot of the $L^\infty$ and $W^{1,\infty}$ errors (left) and the approximate solution (right) by method \eqref{fem} for \ref{itm:P2}, on non-structured meshes.}
\end{figure}
In this example we are considering a curved interface, then quadrature rules over elements with one curved edge are needed. In Appendix \ref{quadrature} we address this issue giving higher-order quadrature rules over this kind of elements. We observe optimal convergence, confirming our results in Theorems \ref{maxgraderrorthm} and \ref{maxerrorthm}. As in the previous example, superconvergence phenomenon is observed in the case of structured meshes.
\end{enumerate}
\subsection{Numerical examples for Stokes interface problem \eqref{StokesProblem}.}
In order to illustrate our method for the Stokes interface problem we consider the pair of inf-sup stable finite element spaces $\n{V}_h$ and $M_h$, piecewise quadratic polynomial for the velocity and piecewise constant polynomials for the pressure ($\mathbb{P}_2^2 - \mathbb{P}_0$). We define the errors
\begin{equation*}
e^{\n{u}}_h := \n{u}_h - I_h \n{u},\qquad e^p_h:= p_h-J_h p,
\end{equation*}

With this finite element spaces, the value of $k$ is 1 (see assumption {\bf A2}) , then $I_h$ is the piecewise linear Lagrange interpolant and $J_h$ is the $L^2$ projection on the spaces of piecewise constant polynomials. We note that the optimal order of convergence for velocity is $2$, this is the reason we are using the piecewise linear Lagrange interpolant although we are using piecewise quadratic in our finite element space.
\begin{enumerate}[label=\bfseries Ex. \theenumi]

\item\label{itm:S1} Consider an exact solution of Stokes problem \eqref{StokesProblem} on $\Omega =[-1,1]$
\begin{align*}
\n{u}(x,y) &= \left(
          \begin{array}{c}
            0 \\
            u_2(x,y) \\
          \end{array}
        \right),\quad u_2(x,y) = \left\{
                            \begin{array}{ll}
                              2/3+x, & \hbox{if } x\leq 1/3,\\
                              4/3-x, & \hbox{if } x>1/3
                            \end{array} \right. ,\\ p(x,y) & = \left\{
                            \begin{array}{ll}
                              x^2+y^2+1/3, & \hbox{if } x\leq 1/3,\\
                              x^2+y^2-8/3, & \hbox{if } x>1/3
                            \end{array}
                          \right. \quad (x,y)\in \Omega.
\end{align*}
In this case the interface is $\Gamma = \{(x,y)\in \Omega : x =1/3\}$.
\begin{table}[H]
\centering
\begin{tabular}{|c|cc|cc|cc|cc|}\hline
$h$ & $\|e^{\n{u}}_h\|_{L^2}$ & $\mbox{r}$& $\|e^{\n{u}}_h\|_{L^\infty}$ & $\mbox{r}$ & $\|\nabla e^{\n{u}}_h\|_{L^2}$ & $\mbox{r}$& $\|\nabla e^{\n{u}}_h\|_{L^\infty}$ & $\mbox{r}$\\ \hline
3.5e-1  &   1.27e-2  &            &    1.22e-2  &           &     1.55e-1 &          &       1.92e-1&\\
1.8e-1  &   3.42e-3  &    1.89    &   3.42e-3   &    1.83   &   7.86e-2   &   0.98   &    1.08e-1    &   0.83\\
8.8e-2  &   8.84e-4  &    1.95    &   9.67e-4   &    1.82   &   3.93e-2   &   1.00   &    5.71e-2    &   0.92\\
4.4e-2  &   2.25e-4  &    1.98    &   2.60e-4   &    1.89   &   1.96e-2   &   1.00   &    3.06e-2    &   0.90\\
2.2e-2  &   5.66e-5  &    1.99    &   6.84e-5   &    1.93   &   9.79e-3   &   1.00   &    1.58e-2    &   0.95\\ \hline
\end{tabular}\vskip2mm
\caption{Errors and orders of convergence for velocity, on structured meshes.}
\end{table}
\begin{table}[H]
\centering
\begin{tabular}{|c|cc|cc|}\hline
$h$ & $\|e^p_h\|_{L^2}$ & $\mbox{r}$& $\|e^p_h\|_{L^\infty}$ & $\mbox{r}$ \\ \hline
3.5e-1   &  6.01e-2  &             &   9.30e-2  & \\
1.8e-1   &  2.22e-2  &    1.44     &  6.07e-2   &    0.62\\
8.8e-2   &  7.45e-3  &    1.58     &  3.52e-2   &    0.79\\
4.4e-2   &  2.34e-3  &    1.67     &  1.92e-2   &    0.88\\
2.2e-2   &  7.09e-4  &    1.73     &  1.00e-2   &    0.93  \\\hline
\end{tabular}\vskip2mm
\caption{Errors and orders of convergence for the pressure, on structured meshes.}
\end{table}
We observe optimal convergence for the velocity and pressure, supporting our result in Theorem \ref{thmStokes}. We also observe a superconvergence phenomenon for the $L^2$ error of the pressure.
\vskip2mm
\item  Consider a exact solution of Stokes problem \eqref{StokesProblem} on $\Omega =[-1,1]$

\begin{align*}
\n{u}(x,y) &= \left(
          \begin{array}{c}
            u_1(x,y) = \left\{
                            \begin{array}{ll}
                              3y & \hbox{if } r\leq 1/3,\\
                             \frac{4y}{3r}-y & \hbox{if } r>1/3
                            \end{array} \right.  \\
            u_2(x,y) =\left\{
\begin{array}{ll}
-3, & \hbox{if } r\leq 1/3,\\
x-\frac{4x}{3r} & \hbox{if } r>1/3
\end{array}\right.\\
          \end{array}
        \right),\\
\quad p(x,y) & = \left\{
\begin{array}{ll}
4-\frac{\pi}{9} & \hbox{if } r\leq 1/3,\\
\frac{\pi}{9} & \hbox{if } r>1/3
\end{array}
\right. ,
\end{align*}
for  $(x,y)\in \Omega$ and $r = \sqrt{x^2+y^2}$.
In this case the interface is the circumference of radius $3$, i.e. $\Gamma = \{(x,y)\in \Omega : r =\frac{1}{3}\}$.

\begin{table}[H]
\centering
\begin{tabular}{|c|cc|cc|cc|cc|}\hline
$h$ & $\|e^{\n{u}}_h\|_{L^2}$ & $\mbox{r}$& $\|e^{\n{u}}_h\|_{L^\infty}$ & $\mbox{r}$ & $\|\nabla e^{\n{u}}_h\|_{L^2}$ & $\mbox{r}$& $\|\nabla e^{\n{u}}_h\|_{L^\infty}$ & $\mbox{r}$\\ \hline
2.5e-01 &    3.02e-2  &          &      3.99e-2  &           &     5.16e-1 &          &       8.10e-1 & \\
1.3e-01 &    8.48e-3  &    1.83  &     1.79e-2   &    1.16   &   2.79e-1   &   0.89   &    5.48e-1    &   0.56\\
6.3e-02 &    2.03e-3  &    2.06  &     5.35e-3   &    1.74   &   1.36e-1   &   1.03   &    3.35e-1    &   0.71\\
3.1e-02 &    5.09e-4  &    2.00  &     1.68e-3   &    1.67   &   6.84e-2   &   0.99   &    2.06e-1    &   0.70\\
1.6e-02 &    1.26e-4  &    2.02  &     4.22e-4   &    1.99   &   3.36e-2   &   1.03   &    1.05e-1    &   0.97\\ \hline
\end{tabular}\vskip2mm
\caption{Errors and orders of convergence for velocity, on structured meshes.}
\end{table}
\begin{table}[H]
\centering
\begin{tabular}{|c|cc|cc|}\hline
$h$ & $\|e^p_h\|_{L^2}$ & $\mbox{r}$& $\|e^p_h\|_{L^\infty}$ & $\mbox{r}$ \\ \hline
2.5e-1  &   1.39e-1  &           &     1.84e-1   & \\
1.3e-1  &   3.39e-2  &    2.04   &    7.71e-2    &   1.26 \\
6.3e-2  &   1.32e-2  &    1.36   &    4.29e-2    &   0.85 \\
3.1e-2  &   3.79e-3  &    1.80   &    2.36e-2    &   0.86 \\
1.6e-2  &   1.46e-3  &    1.38   &    1.27e-2    &   0.90    \\\hline
\end{tabular}\vskip2mm
\caption{Errors and orders of convergence for pressure, on structured meshes.}
\end{table}
\begin{figure}[H]\label{Stokesfigure}
\includegraphics[scale=.3]{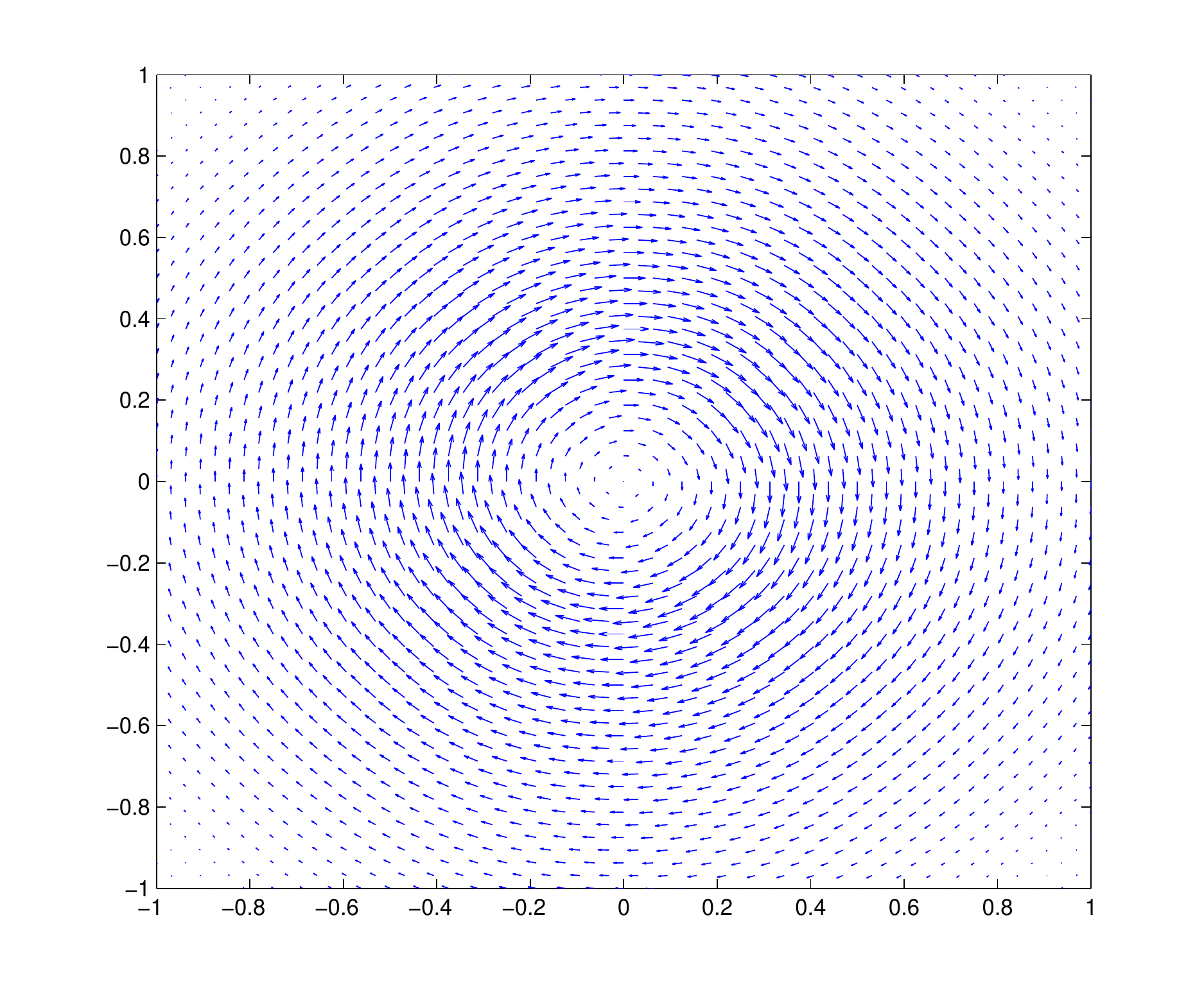}\includegraphics[scale=.425]{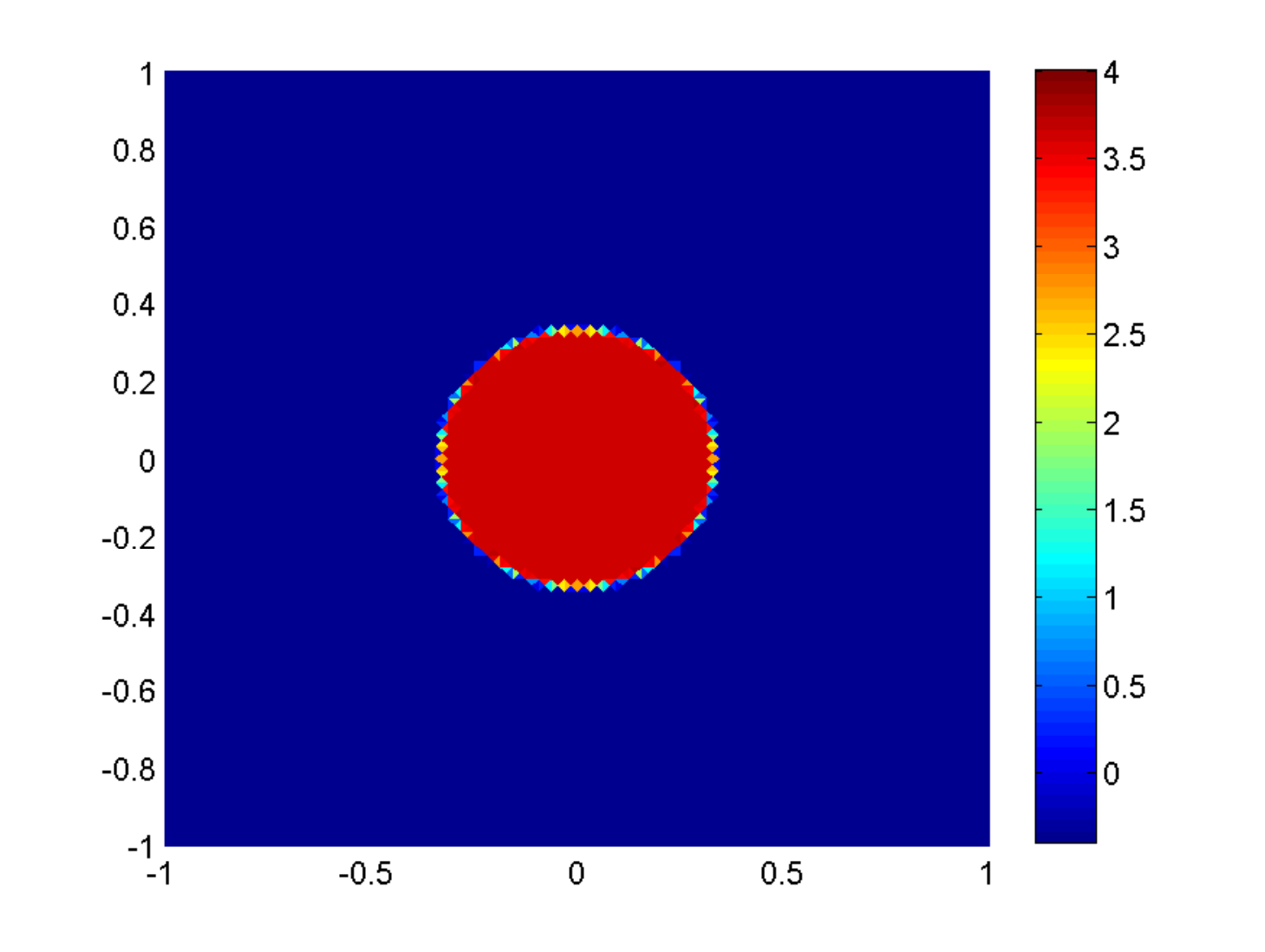}\vskip2mm
\caption{Plot of velocity (left) and discontinuous pressure (right).}
\end{figure}
We observe optimal convergence for the velocity and pressure, confirming our result in Theorem \ref{thmStokes}. We also observe from the plot of the approximate pressure (see Figure \ref{Stokesfigure}) that there is not pollution beyond the interface.

\end{enumerate}

\section{Final discussion and future work}
In this paper we have developed higher-order finite element methods for Poisson interface problem \eqref{Problem} and Stokes interface problem \eqref{StokesProblem}. We have proved optimal convergence results for the methods, recovering the convergence of the finite element method for a problem without interface. We have presented numerical experiments validating our theoretical results. The stiffness matrix remains the same as the problem without the interface. To the best of our knowledge, this is the first family of numerical methods that have provable higher-order accuracy (or arbitrary order) for these problems.  

The algorithms and theory developed point to a number of promising research directions and
opportunities. In the future, we plan to use our techniques to develop and analyze  high contrast
problems for elliptic problems with discontinuous diffusion coefficients. Additionally, it is natural
to move in the direction of evolution problems, such as the time-dependent version of
problem \eqref{StokesProblem} with moving interfaces. We plan to develop time stepping methods where in each time step we used the methodology developed here to solve the resulting problem.  Navier-Stokes equations and more challenging problems such as fluid-structure interactions will be explored in the future.

Finally, more straightforward generalizations are possible. In this paper, for convenience, we used triangular meshes, but it is clear that rectangular or quadrilateral meshes can be used. Indeed, the correction function did not use the fact that the mesh is triangular. Three dimensional problems should not pose a serious difficulty. The key to our method is the correction functions $w_T^u$. The construction of corresponding correction functions in three dimensions seems natural.  Details will appear in the Ph.D. thesis of the second author.

\vskip4mm
\appendix
\section{Proof of estimate \eqref{Dngbound}}\label{AppendixDngbound}

We first use a dyadic decomposition of $\Omega$. $\Omega = \Omega^* \cup \bigcup_{j=1}^J \Omega_j\,\,$ where $\,\,\Omega_j  := \{ x \in \Omega: d_{j+1} \le |x-z| \le  d_{j} \}$, and $ d_j=2^{-j}$. Here, $J=\log_2(1/h)$  and $\Omega^*= \Omega \cap B_h$ where $B_h$ is the ball of radius $h$ centered $z$.

Let us use the notation $D_j=  \Omega_j \cap \partial R_h \backslash \partial \Omega$ and $D^*=\Omega^* \cap \partial R_h \backslash \partial \Omega$ then we have
\begin{equation*}
\| D_{\n{n}}  g\|_{L^1(\partial R_h \backslash \partial \Omega)} =\| D_{\n{n}}  g\|_{L^1(D^*)}+ \sum_{j=1}^J \|D_{\n{n}}  g\|_{L^1(D_j)}
\end{equation*}

First we bound the first term. Using Cauchy-Schwarz inequality and the fact the one dimensional  measure of $D^*$ is $h$ we have
\begin{equation*}
\| D_{\n{n}}  g\|_{L^1(D^*)} \le C h^{1/2} \| \nabla g\|_{L^2(D^*)}.
\end{equation*}
If we use a trace inequality then we have

\begin{equation*}
h^{1/2} \| \nabla g\|_{L^2(D^*)} \le C ( \| \nabla g\|_{L^2(\Omega^* \cap B_h)}+ h \|D^2 g\|_{L^2(\Omega^* \cap B_h)}) \le C (\| \nabla g\|_{L^2(\Omega)}+ h \|D^2 g\|_{L^2(\Omega)}).
\end{equation*}
Elliptic regularity will give
\begin{equation*}
\| \nabla g\|_{L^2(\Omega)}+ h \|D^2 g\|_{L^2(\Omega)} \le C( \|\delta_h\|_{L^2(\Omega)} + h \|\delta_h\|_{H^1(\Omega)}) \le \frac{C}{h},
\end{equation*}
which shows that
\begin{equation*}
\| D_{\n{n}}  g\|_{L^1(D^*)} \le \frac{C}{h}.
\end{equation*}

To take care of the sum we use Cauchy-Schwarz inequality again and get
\begin{equation*}
\sum_{j=1}^J \| D_{\n{n}}  g\|_{L^1(D_j)} \le \sum_{j=1}^J d_j \| \nabla g\|_{L^\infty(D_j)}
\end{equation*}

Using the Green's function representation of $g$ and the fact that $D_j$ is distance $d_j$ away from $z$ one can show (see \cite{GSS2014})
\begin{equation*}
\|\nabla g\|_{L^\infty(D_j)} \le \frac{C}{d_j}.
\end{equation*}

Hence,
\begin{equation*}
\sum_{j=1}^J \| D_{\n{n}}  g\|_{L^1(D_j)} \le C \sum_{j=1}^J \frac{1}{d_j}=C  \sum_{j=1}^J \frac{1}{d_j}= C (2^{J+1}-1)\le \frac{C}{h}.
\end{equation*}

\section{Derivation of jumps \eqref{JumpsStokes}}\label{Ajumps}

We define for $\epsilon>0$ the set $\Omega_\epsilon = \{x\in \Omega:\quad dist(x,\Gamma) < \epsilon\}$, where $dist $ is the standard distance function. We will also denote $\Gamma_\epsilon^\pm = \partial \Omega_\epsilon\cap \Omega^\pm$. From equation \eqref{StokesProblem2:a}, taking divergence, multiplying by an arbitrary $\phi\in C^2$ of compact support in $\Omega_\epsilon$  and integrating we obtain
\begin{align}\label{deltap}
\nonumber \int_{\Omega_\epsilon} \nabla\cdot \nabla p\, \phi dx &= \int_{\Omega_{\epsilon}}\nabla\cdot f\phi dx + \int_{\Omega_{\epsilon}} \nabla \cdot \n{B}\phi dx \\
                                           &= \int_{\Omega_{\epsilon}}\nabla\cdot f \phi dx - \int_{\Gamma} \n{\beta}\cdot \nabla \phi ds.
\end{align}

For the left-hand side of the equation above we integrate by parts twice obtaining
\begin{align*}
\int_{\Omega_\epsilon} \nabla\cdot \nabla p \,\phi dx &= \int_{\Gamma_\epsilon^+} D p^+\cdot n \phi ds + \int_{\Gamma_\epsilon^-} D p^-\cdot n \phi ds \\
                                                    & -\int_{\Gamma_\epsilon^+} D \phi\cdot n p^+ ds -\int_{\Gamma_\epsilon^-} D \phi\cdot n p^- ds + \int_{\Omega_\epsilon} p  \nabla\cdot \nabla \phi dx.
\end{align*}
By the definitions of $\Gamma_\epsilon^\pm $ we can fix its normal vectors by the normal vector to the interface $\Gamma$.
Then, taking the limit as $\epsilon$ goes to zero we obtain
\begin{align}\label{limitp}
\lim_{\epsilon\rightarrow 0} \int_{\Omega_\epsilon} p  \nabla\cdot \nabla \phi dx& = 0 \qquad  \\
\lim_{\epsilon\rightarrow 0}\int_{\Omega_\epsilon} \nabla\cdot \nabla p \phi dx & \quad = \quad  \int_{\Gamma} \left[D p\cdot n\right] \phi ds -\int_{\Gamma} D \phi\cdot n \left[ p \right] ds
\end{align}

We need to write the right-hand side in \eqref{deltap} in terms of the normal and tangential derivative. Let $\theta$ be the angle between the $x$ direction ($x$-axis) and $n$ direction and let $R(\theta)$ the rotation matrix defined by
\begin{equation*}
R(\theta) = \left(
  \begin{array}{cc}
    \cos(\theta) & -\sin(\theta) \\
    \sin(\theta) & \cos(\theta) \\
  \end{array}
\right)
\end{equation*}

Then, $\nabla \phi = (D_{\n{n}}\phi,D_{\n{t}}\phi) R(\theta)^t $. Using this equivalency in \eqref{deltap} and integration by parts we obtain
\begin{align*}
\int_{\Gamma} \n{\beta}\cdot \nabla \phi ds & = \int_{\Gamma} \n{\beta} \cdot \left((D_{\n{n}}\phi,D_{\n{t}}\phi) R(\theta)^t\right) ds \\
                                            & = \int_{\Gamma} \left( \n{\beta} R(\theta)\right)\cdot  (D_{\n{n}}\phi,D_{\n{t}}\phi)  ds \\
                                            & = \int_{\Gamma} \hat{\n{\beta}} \cdot (D_{\n{n}}\phi,D_{\n{t}}\phi)   ds\\
                                            & = \int_{\Gamma} \hat{\beta}_1 D_{\n{n}}\phi  ds -  \int_{\Gamma} \frac{\partial}{\partial s}\hat{\beta}_2 \phi   ds
\end{align*}
Here we have defined $\n{\hat{\beta}} = \n{\beta} R(\theta)$. Matching this result with the limits in \eqref{limitp} we can conclude that
\begin{align}
\left[p\right] \, = \,\hat{\beta}_1 &\qquad
\left[D_{\n{n}} p \right] \, = \, \frac{\partial}{\partial s}\hat{\beta}_2.
\end{align}

Now, to get the normal jump of the gradient of the velocity we applied equation \eqref{StokesProblem:d}, resulting
\begin{equation}
\left[D_{\n{n}} \n{u} \right] = \n{\beta} + \left[p \n{n}\right] =  \n{\beta} +  \hat{\beta}_1 \n{n}.
\end{equation}

\section{Quadrature over curved elements.}\label{quadrature}

The method defined in \eqref{fem} requires the integration, over an element $T\in \mathcal{T}_h^\Gamma$, of $\nabla w_T^u \cdot \nabla v$. Now, since the correction function $w_T^u$ is defined as piecewise polynomial of degree $k$ on each piece of the restriction $T^\pm = T\cap \Omega^\pm$, we need some quadrature formulas on this region, say curved elements (polygons with one edge curved). To do so,
we observe that $L_T$ divides $T$ in two regular polygons $\tilde{T}^\pm$, one is a triangle and the other a quadrilateral, see
figure \ref{fig:quad}. In order to make
the presentation clearer, we will write the integral in terms of the coordinates system $(\tau,\,\eta)$ (associated to the vector $\n{\tau}$ and $\n{\eta}$, with $\n{\eta}$ pointing outward $T^-$) such that the origin is the midpoint of $L_T$, the point $\frac{y_T+z_T}{2}$. Let $f\in C^{n+1}(T)$,  then we write

\begin{equation} \label{integral}
\int_{T^{\pm}} f(x,y) d{x} d{y} \,\,=\,\, \int_{\tilde{T}^{\pm}} f(x,y) d{x} d{y} \pm \int_{-r_T/2}^{r_T/2} \int_{0}^{\gamma(\tau)} f(\tau,\eta) d{\eta} d{\tau}
\end{equation}
where, $\gamma(\tau)$ is the equation for the curve $\Gamma$ on $T$. The first integral in \eqref{integral} is over a polygon (triangle or quadrilateral) and
we can use quadrature rules over polygons to approximate it, whereas the second integral is related to the curve, and we will use some Gaussian quadratures.
First step, we take $n$ Gaussian points $\left\{\tau_i\right\}_{i=1}^n$ over the segment $L_T$, equivalently, the interval $[-r_T/2, r_T/2]$ in the $(\tau,\,\eta)$-axis, and for each point we compute its projection onto the curve $\Gamma$ in the $\n{\eta}-$direction, say $\gamma(\tau_i)$. Now we compute, the Gaussian points for the segment $[0,\gamma(\tau_i)]$, we denote them by $\{\eta_{i,j}\}_{j=1}^n$. Therefore
\begin{align*}
\int_{-r_T/2}^{r_T/2} \int_{0}^{\gamma(\tau)} f(\tau,\eta) d{\tau} d{\eta}\,\,\approx\,\,\sum_{i=1}^n \frac{r_T}{2}\omega_i \sum_{j=1}^n \frac{|\gamma(\tau_i)|}{2} \omega_j f(\tau_i,\eta_{i,j})
\end{align*}
where $\omega_i$ are the weights in the interval $[-1,1]$. See figure \ref{fig:quad} for illustration of this notation. See also figure \ref{fig:Fig1}.

\begin{figure}[H]
\includegraphics[scale=.4]{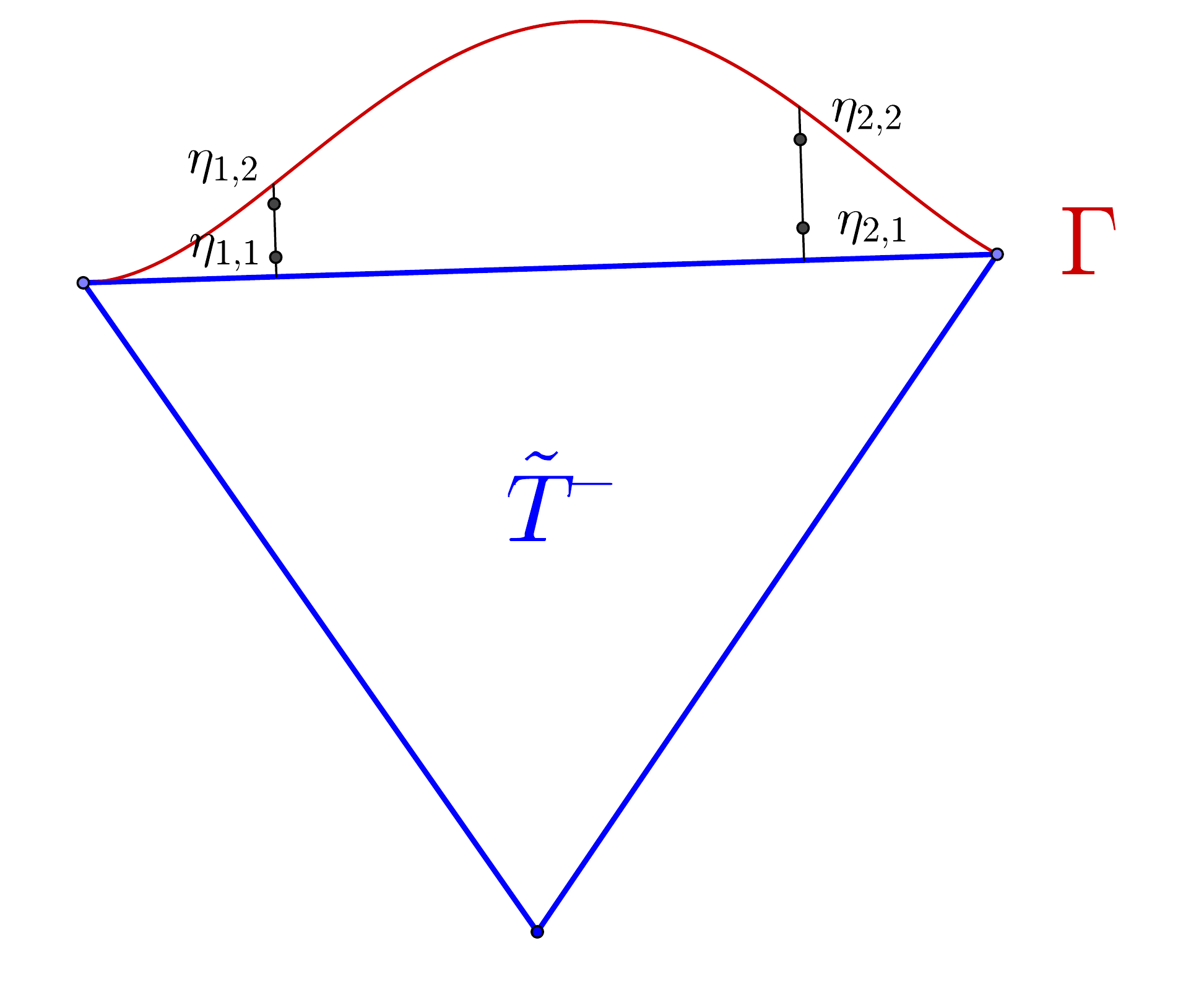}
\caption{Illustration of the Gauss points.}
\label{fig:quad}
\end{figure}

\bibliography{HigherorderBiblio}{}
\bibliographystyle{plain}
\end{document}